\documentclass[a4paper,10pt]{amsart}

\usepackage[utf8]{inputenc}
\usepackage[T1]{fontenc}
\usepackage{amsfonts}
\usepackage{amsthm}
\usepackage{amsmath}
\usepackage[english]{babel}
\usepackage[all]{xy}
\usepackage{tikz-cd}
\usepackage{wrapfig}

\usepackage{ amssymb }
\usepackage{graphicx}
\usepackage{amscd}
\usepackage{latexsym}
\usepackage{hyperref}
\usepackage{mathrsfs}
\usepackage{enumerate}

\graphicspath{{pictures/}}
\usepackage{xcolor}

\usepackage{calc}

\newtheorem{theoremABC}{Theorem}

\newtheorem{conjectureABC}[theoremABC]{Conjecture}

\theoremstyle{definition}
 \newtheorem{defi}{Definition}[section]

\theoremstyle{remark}
 \newtheorem{remark}[defi]{Remark}

\theoremstyle{plain}

\newtheorem{prop}[defi]{Proposition}

\newtheorem{lemma}[defi]{Lemma}

\newtheorem*{acknow}{Acknowledgments}

\def\e#1\e{\begin{equation}#1\end{equation}}
\def\ea#1\ea{\begin{align}#1\end{align}}

\numberwithin{equation}{section}

\newcommand{\zz}{\mathbb{Z}}

\newcommand{\rr}{\mathbb{R}}

\newcommand{\Z}[1]{\zz_{#1}}

\newcommand{\abs}[1]{\left\vert #1 \right\vert}

\newcommand{\Acal}{\mathcal{A}}
\newcommand{\Bcal}{\mathcal{B}}
\newcommand{\Ccal}{\mathcal{C}}
\newcommand{\Dcal}{\mathcal{D}}

\newcommand{\Fcal}{\mathcal{F}}
\newcommand{\Gcal}{\mathcal{G}}

\newcommand{\Mcal}{\mathcal{M}}
\newcommand{\Ncal}{\mathcal{N}}

\newcommand{\Wcal}{\mathcal{W}}
\newcommand{\Xcal}{\mathcal{X}}

\newcommand{\Lbb}{\mathbb{L}}

\newcommand{\Uund}{\underline{U}}
\newcommand{\Dund}{\underline{D}}

\newcommand{\Lund}{\underline{L}}

\newcommand{\kund}{\underline{k}}
\newcommand{\lund}{\underline{l}}

\newcommand{\xund}{\underline{x}}

\renewcommand{\Vert}{\mathrm{Vert}}
\newcommand{\Edges}{\mathrm{Edges}}
\newcommand{\IntEd}{\mathrm{IntEd}}
\newcommand{\Leaves}{\mathrm{Leaves}}
\newcommand{\Roots}{\mathrm{Roots}}

\newcommand{\Ainf}{$A_\infty$}

\newcommand{\One}{\underline{1}}

\newcommand{\Fuk}{\mathcal{F}uk}

\newcommand{\Symp}{\mathcal{S}ymp}

\newcommand{\Ham}{\mathcal{H}am}

\newcommand{\fBialg}{f\text{-}\mathcal{B}ialg}

\newcommand{\dCat}{d\text{-}\mathcal{C}at}

\newcommand{\Cob}{\mathcal{C}ob}

\newcommand{\arnaque}{\begin{flushright}
$\Box$
\end{flushright}}

\newcommand{\OSz}{Ozsv{\'a}th and  Szab{\'o}}
\newcommand{\MW}{Manolescu and Woodward}
\newcommand{\WW}{Wehrheim and Woodward}

\title[Bialgebras, and Lie monoid actions in Morse and Floer theory, I]{Bialgebras, and Lie monoid actions in Morse and Floer theory, I}

\author{Guillem Cazassus}
\address{Centre for Quantum Mathematics, 
University of Southern Denmark}
\email{g.cazassus@gmail.com}

\author{Alexander Hock}
\address{Institut für Mathematik, Heidelberg University}
\email{alexander.hock@uni-heidelberg.de}

\author{Thibaut Mazuir}
\address{Institut für Mathematik, Humboldt Universität zu Berlin}
\email{thibaut.mazuir@hotmail.fr}

\thanks{GC was funded by EPSRC grant reference EP/T012749/1, the Simons Collaboration grant no. 994320, and the ERC-SyG project ReNewQuantum. AH’s work is funded by the German Research Foundation (Deutsche Forschungsgemeinschaft, DFG) through a Walter-Benjamin fellowship, project number 551922044. TM was funded by the Deutsche Forschungsgemeinschaft (DFG, German Research Foundation) under Germany's Excellence Strategy – The Berlin Mathematics Research Center MATH+ (EXC-2046/1, project ID: 390685689).}
\begin{document}

\begin{abstract}
We introduce a new family of oriented manifolds with boundaries called the forest biassociahedra and forest bimultiplihedra, generalizing the standard biassociahedra. 
They are defined as moduli spaces of ascending-descending biforests and are expected to act as parameter spaces for operations defined on Morse and Floer chains in the context of compact Lie group actions.

We study the structure of their boundary, and derive some algebraic notions of ``$f$-bialgebras'', as well as related notions of bimodules, morphisms and categories. This allows us to state some conjectures describing compact Lie group actions on Morse and Floer chains, and on Fukaya categories.

\end{abstract}

\maketitle
\tableofcontents

\newpage

\section{Introduction}
\label{sec:intro}

Consider a compact Lie group $G$, or more generally a Lie monoid\footnote{by which we mean a closed smooth manifold with an associative product}, acting on a smooth manifold $X$ (resp. a symplectic manifold $M$). How does its action interact with  the multiplicative structures of Morse theory of $X$ (resp. Floer theory of $M$)? 
The goal of this paper is to introduce a new ad hoc algebraic framework  in order to answer this question.

Our construction starts with the following simple observation. The homology $H_*(G)$ has the structure of a bialgebra (and if $G$ is a group, a Hopf algebra): its product is the Pontryagin product, which captures the group structure and is induced by the multiplication; and its coproduct is dual to the cup product and induced by the diagonal $G\to G\times G$, which is the classical counterpart of the Donaldson product (the multiplicative structure in Floer theory).

We produce a chain-level counterpart of this structure, called \emph{forest bialgebra}, or \emph{f-bialgebra}, that appears naturally on Morse chains. We use the word ``forest'' in order to emphasize that this structure is defined in terms of combinatorics of forests and not only trees as for standard algebraic operads.

By studying some moduli spaces of metric ``biforests'', we introduce 
some algebraic notions that lead us to a rigorous algebraic framework allowing us to formulate the $G$-action at the chain level. As forests are involved, operations will be indexed by multi-indices $\kund = (k_1, \ldots , k_a)$, $\lund = (l_1, \ldots , l_b)$, where $k_i, l_i\geq 1$ are integers, as opposed to just pairs of integers. 
The various notions in the conjectures below will be introduced in Section~\ref{sec:bialg_bimod_mph}. Loosely speaking, these are all chain-level notions that, on the level of homology, induce the standard notions suggested by their names. The prefix $u$- (resp. $d$-) stands for ``up'' (resp. ``down''), and refers to ascending forests corresponding to $\kund$, related to the product structure (resp. descending forests corresponding to $\lund$, related to the coproduct structure).

\begin{conjectureABC}[Lie monoid actions in Morse theory]\label{conj:Morse}
Assume that:

\noindent\begin{minipage}{.2\textwidth}\begin{tikzcd}
   G \ar{r}{\varphi}& G' \\
   X \ar{r}{f} \ar[,loop ,out=123,in=57,distance=2.5em]{}{} \ar[,loop ,out=-123,in=-57,distance=2.5em]{}{}& X' \ar[,loop ,out=123,in=57,distance=2.5em]{}{} \ar[,loop ,out=-123,in=-57,distance=2.5em]{}{}\\
   H \ar{r}{\psi}& H'
\end{tikzcd}
\end{minipage}
\begin{minipage}{.7\textwidth}
\begin{itemize}
\item $G,G',H,H'$ are compact Lie monoids,
\item $X,X'$ are closed smooth manifolds, acted on by $G\times H$ and $G'\times H'$ respectively,
\item $\varphi, \psi$ are morphisms of Lie monoids,
\item $f\colon X \to X'$ is $(\varphi, \psi)$ bi-equivariant.
\end{itemize}
\end{minipage}

Then, at the level of Morse complexes\footnote{When writing the Morse complex of a manifold, we leave implicit some auxiliary choices (Morse functions, perturbations, ...)}:
\begin{enumerate}
\item $CM(G),CM(G'),CM(H),CM(H')$ are f-bialgebras,
\item $CM(X),CM(X')$ are $u$-bimodules (of f-bialgebras),
\item $\varphi, \psi$ induce morphisms of f-bialgebras 
\ea
\varphi_* &\colon CM(G)\to CM(G'), \\
\psi_* &\colon  CM(H)\to CM(H').
\ea
\item $f$ induces a $(\varphi_*, \psi_*)$ bi-equivariant morphism of $u$-bimodules 
\e
f_*\colon CM(X)\to CM(X').
\e
\end{enumerate}
\end{conjectureABC}

We now state the Floer theoretic counterpart of Conjecture~\ref{conj:Morse}. We loosely define a \emph{nice enough} symplectic manifold or Lagrangian to be a symplectic manifold or Lagrangian satisfying some standard assumptions that prevent bubbling (exact, monotone...). We say that a Lagrangian correspondence $\Lambda\subset M^-\times M'$ is \emph{well-composable on the left} if for any nice enough $L\subset M$, the composition $L\circ\Lambda\subset M'$ is nice enough, and embedded in the sense of Wehrheim and Woodward \cite{WWcompo}.

\begin{remark}We choose to work with \emph{chain} complexes, as opposed to \emph{cochain} complexes, in order to work in a covariant setting. In particular a group action leads to a module structure at the chain level, as opposed to a comodule structure. The Fukaya category, as ordinarily defined, involves Floer \emph{cochains}. Therefore, what used to be composition operations become ``decomposition operations'' in our setting, i.e. one gets an \Ainf\ ``co-category''. One can easily get dual similar statements for cochains by exchanging the roles of $\kund$ and $\lund$.
\end{remark}

\begin{conjectureABC}[Lie monoid actions in Floer theory]\label{conj:Floer}
Assume that:

\noindent\begin{minipage}{.3\textwidth}\begin{tikzcd}
 &  G \ar{r}{\varphi}& G' \\
L_0, L_1 \ar[hookrightarrow]{r}{}&  M \ar{r}{\Lambda} \ar[,loop ,out=123,in=57,distance=2.5em]{}{} \ar[,loop ,out=-123,in=-57,distance=2.5em]{}{}& M' \ar[,loop ,out=123,in=57,distance=2.5em]{}{} \ar[,loop ,out=-123,in=-57,distance=2.5em]{}{}\\
 &  H \ar{r}{\psi}& H'
\end{tikzcd}
\end{minipage}
\begin{minipage}{.7\textwidth}
\begin{itemize}
\item $G,G',H,H',\varphi, \psi$ are as above.
\item $M,M'$ are nice enough symplectic manifolds, acted on symplectically by $G\times H$ and $G'\times H'$ respectively,
\item $L_0, L_1\subset M$ is a pair of $(G\times H)$-invariant nice enough Lagrangians,
\item $\Lambda\subset M^-\times M'$ is a $(G\times H)$-invariant Lagrangian correspondence well-composable on the left.
\end{itemize}
\end{minipage}
Then:
\begin{enumerate}
\item $CF(L_0), CF(L_1)$ are $u$-bimodules,
\item $CF(L_0, L_1)$ is a $\square$-quadrimodule, with
\e
\square = \begin{pmatrix}
CM(G) & CM(G) & CM(G) \\ 
CF(L_0) & . & CF(L_1) \\ 
CM(H) & CM(H) & CM(H) 
\end{pmatrix} .
\e
\item The sub-categories of the \emph{Fukaya co-categories}  $\Fuk(M), \Fuk(M')$ consisting in invariant Lagrangians can be upgraded to $u$-bimodule $d$-categories over $(CM(G), CM(H))$, resp. $(CM(G'), CM(H'))$.
\item The \Ainf -functor  $\Phi_{\Lambda}\colon \Fuk(M)\to \Fuk(M')$ \cite{MauWehrheimWoodward,Fukaya_functor} can be upgraded to a $(\varphi_*, \psi_*)$-equivariant functor of $u$-bimodule $d$-categories.
\end{enumerate}
\end{conjectureABC}

\begin{remark}In Conjecture~\ref{conj:Floer}, (1) and (2) are consequences of (3).
\end{remark}

\begin{remark}As for non-equivariant Floer theory, if one wants to allow more general symplectic manifolds and (possibly immersed) Lagrangians, and drop the ``nice enough'' and ``well-composable at the left'' assumptions, one must work with curved structures. Curved analogs in our setting should also exist, and should include operations for which entries of the multi-index $\lund$ can be zero. We expect a (by no means straightforward) parallel story as in 
\cite{FOOObook1,FOOObook2}, with bounding chains, etc... 
\end{remark}

In the setting of Hamiltonian actions, one can replace Morse complexes by wrapped Fukaya categories. This enables extending the class of objects to possibly non-invariant Lagrangians, and is perhaps more closely related to Teleman's approach, see Section~\ref{ssec:rel_work}. Since composition with the Weinstein correspondence might lead to immersions, one needs to allow immersions here, and therefore ``$d$-curved structures''.

\begin{conjectureABC}[Hamiltonian actions in Floer theory]\label{conj:wrapped}
Assume now that in the setting of Conjecture~\ref{conj:Floer}  the actions are Hamiltonian, then:

\begin{enumerate}
\item The wrapped Fukaya categories $\Wcal (T^* G), \Wcal (T^* H)$ are $d$-\emph{curved} $d$-categories.
\item The morphisms $\varphi, \psi$ induce $d$-functors
\ea
\varphi_* &\colon \Wcal (T^* G) \to \Wcal (T^* G'), \\
\psi_*  &\colon  \Wcal (T^* H)\to  \Wcal (T^* H').
\ea
\item  The whole co-category $\Fuk(M)$ (including non-invariant Lagrangian immersions) is a $d$-\emph{curved} $(\Wcal (T^* G), \Wcal (T^* H))$ $u$-bimodule $d$-categories.
\item The \Ainf -functor  $\Phi_{\Lambda}\colon \Fuk(M)\to \Fuk(M')$ can be upgraded to a $(\varphi_*, \psi_*)$-equivariant functor of $u$-bimodule $d$-categories.
\end{enumerate}
\end{conjectureABC}

The first author plans to prove Conjecture~\ref{conj:Morse} and \ref{conj:Floer} with due details in forthcoming work \cite{morse_fun,biass2}. In Section\ref{ssec:idea_constr} and throughout the paper we have included some hints for the Morse and Floer experts that connect our constructions with these conjectures. 
Except for these hints, we hope that this paper is accessible without prior knowledge in Morse and Floer theory.

\begin{remark}These statements are stated  for Lie monoids, but if these are Lie groups, $f$-bialgebras can be promoted to ``$f$-Hopf algebra'', a notion that will be introduced in \cite{f_Hopf}.
\end{remark}

As a first step towards Conjectures~\ref{conj:Morse}, \ref{conj:Floer} and \ref{conj:wrapped}, we introduce in this paper new moduli spaces of metric ``biforests'' $K^{\kund}_{\lund}$ and $J^{\kund}_{\lund}$, called respectively  \emph{forest biassociahedra} and  \emph{forest bimultiplihedra}. They will be used as as parameter spaces in order to define some moduli spaces of flowlines and pseudo-holomorphic curves \cite{biass2}. We construct ``face-coherent partial compactifications'', i.e. we glue to them products of copies of themselves and obtain manifolds with boundaries $\left( \overline{K}^{\kund}_{\lund} \right)_{\leq 1}$ and $\left( \overline{J}^{\kund}_{\lund} \right)_{\leq 1}$. These are noncompact, but are ``compact enough'' for our purposes, meaning that generically, the one dimensional moduli spaces they parametrize will extend to compact moduli spaces over them. The precise description of their boundary will then allow us to introduce the algebraic notions one should observe in Morse and Floer theory. We conjecture that these new manifolds with boundaries can in fact be realized as standard polytopes in a Euclidean space. 
We moreover provide five different geometric realizations  (Section~\ref{ssec:geom_realiz}), that should allow one to prove our various conjectures.

\subsection{Idea of the constructions}
\label{ssec:idea_constr}We now present the main ideas of our approach, and explain why Conjectures~\ref{conj:Morse}, \ref{conj:Floer} and  \ref{conj:wrapped} should hold. For simplicity, let us work over $\Z{2}$ in this section. 
We focus on a Lie monoid $G$ endowed with a Morse function (implicit in the notations). 
The Pontrjagyn product and the coproduct are respectively induced by the multiplication and the diagonal:
\ea
m &\colon G\times G \to G, \\
\Delta &\colon G \to G\times G.
\ea
They respectively induce an \Ainf -algebra and an \Ainf -coalgebra structure on the Morse chains $CM(G)$. The \Ainf -coalgebra structure (more precisely, the corresponding dual \Ainf -algebra structure on cochains) is due to Fukaya \cite{Fukaya_htpy}, see also \cite{FukayaOh_zeroloop,Mescher_book}. It is defined by counting Morse flow trees. The \Ainf -algebra structure was introduced in \cite{equiv_trees} and involves similar trees, except that they grow in the opposite direction, and that the condition at vertices involves $m$ rather than $\Delta$.

One can ask how these two structures interact. Are they part of a larger structure? At the homology level, the product and coproduct satisfy the Hopf relation:
\e\label{eq:Hopf}
\Delta_* \circ m_* = (m_*)^{\otimes 2} \circ \tau_{23} \circ (\Delta_*)^{\otimes 2} ,
\e
with $\tau_{23}\colon H_*(G)^{\otimes 4} \to H_*(G)^{\otimes 4}$ exchanging the second and third factors.

At the chain level, it is only expected to hold up to homotopy. Indeed, one can construct a map ${\alpha}^2_2 \colon CM(G)^{\otimes 2} \to CM(G)^{\otimes 2} $ by counting a moduli space of flow graphs, parametrized by a family of abstract graphs $K_2^2$, as in Figure~\ref{fig:K_2_2}. Denoting ${\alpha}^2_1=m_*$ and ${\alpha}^1_2=\Delta_*$ the chain-level pushforwards, the relation becomes (ignoring signs):
\e\label{eq:R_2_2}\tag{$R^2_2$}
\partial ( {\alpha}^2_2 ) = {\alpha}^1_2 \circ {\alpha}^2_1 + ({\alpha}^2_1)^{\otimes 2} \circ \tau_{23} \circ ({\alpha}^1_2)^{\otimes 2} ,
\e
which, as one expects, is (\ref{eq:Hopf}) up to homotopy. In order to understand what further coherence relations these satisfy, one is led to consider a 2-dimensional family of graphs as drawn in Figure~\ref{fig:K_3_2}.

\begin{figure}[!h]
    \centering
    \def\svgwidth{.50\textwidth}
    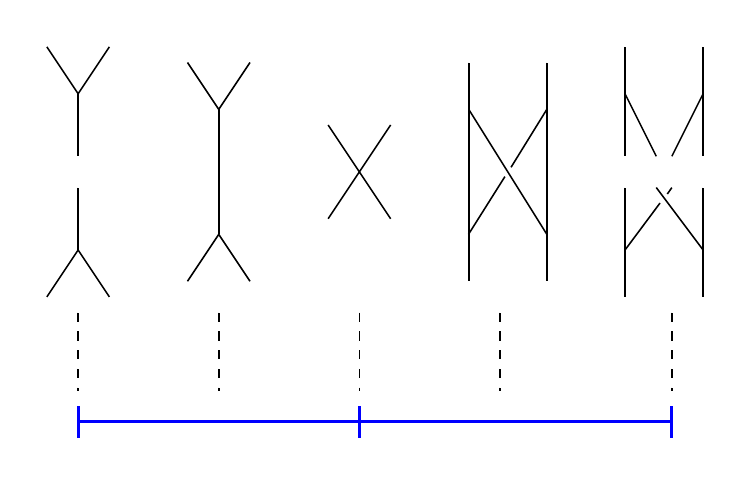
      \caption{The moduli space $K^2_2$. It consists in two intervals, parametrized by the height between vertices. These two intervals overlap where these heights are zero.}
      \label{fig:K_2_2}
\end{figure}

\begin{figure}[!h]
    \centering
    \def\svgwidth{.70\textwidth}
    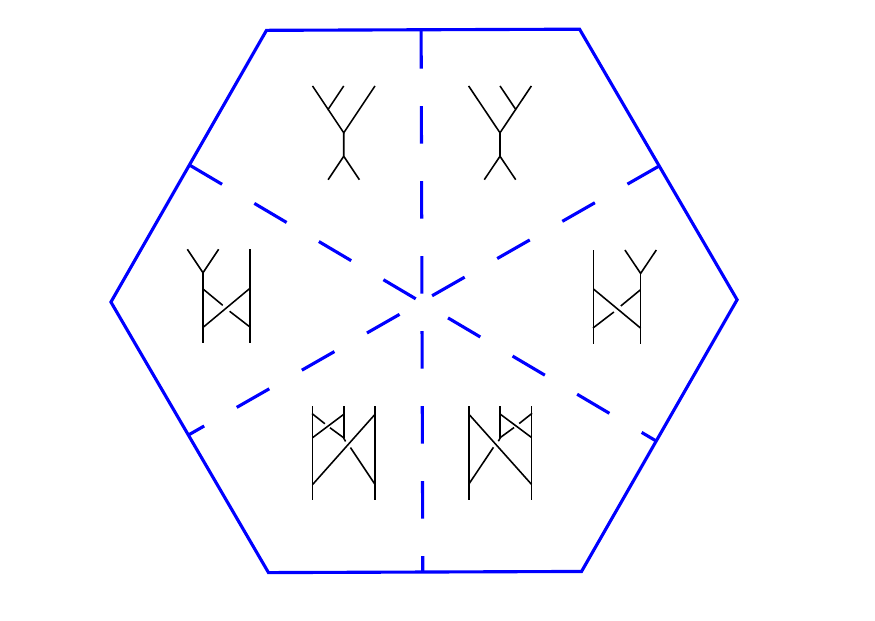
      \caption{The moduli space $K^3_2$. Each of the six chambers is parametrized by two heights. When a height goes to zero, one crosses a wall (interior dashed lines) and moves to an adjascent chamber. When a height goes to infinity, the graph breaks: one reaches an exterior boundary of the hexagon.}
      \label{fig:K_3_2}
\end{figure}

This suggests that more generally, there should exist moduli spaces $K^k_l$ encoding the combinatorics of similar graphs with $k$ inputs and $l$ outputs. These should contain $K_k \times K_l \times [0, +\infty)$ by gluing an ascending $k$-tree with a descending $l$-tree, with distance between the root vertices given by the $[0, +\infty)$ parameter. When this parameter goes to zero, this family should continue, with the two merging vertices replaced by a ``Hopf pattern'' (as in $K^2_2$), and the new born vertices might also collide with other vertices, giving further Hopf patterns as in the two bottom chambers of $K^3_2$. In this process, graphs become increasingly complicated, and it 
 was not so clear to us how to describe these in full generality, until André Henriques made a beautiful observation to us: ``your graphs are obtained by intersecting trees'', i.e. as in Figure~\ref{fig:inters_graphs}. This observation seems to have been known to combinatorics experts as well, as we were told by Bruno Vallette and Frédéric Chapoton. 

\begin{figure}[!h]
    \centering
    \def\svgwidth{.50\textwidth}
    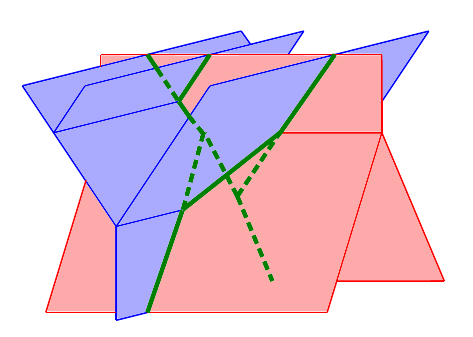
      \caption{Graphs as intersections of trees. Observe that the green graph corresponds to the middle left chamber of $K^3_2$.}
      \label{fig:inters_graphs}
\end{figure}

It now becomes clear that one can define $K^k_l$ as :
\e
K^k_l = K_k \times K_l \times \rr,
\e
with the $\rr$ parameter corresponding to the relative height between the root vertices, allowed to be negative. This makes the relation with the other biassociahedron more transparent, see Section~\ref{ssec:rel_work}.

By applying the above intersection procedure, we can define moduli spaces of flow graphs, and a family of maps 
\e
{\alpha}^k_l\colon CM(G)^k \to CM(G)^l ,
\e
defined by counting their zero dimensional part. To write down which coherence relations these satisfy, as it is standard in Morse theory, one considers one-dimensional moduli spaces and use the fact that the signed count of points in the boundary of their compactification is zero. Two kinds of contributions appear: 
\begin{itemize}
\item breaking at external edges, contributing to the homotopy summands $\partial ({\alpha}^k_l)$,
\item domain degenerations, where lengths of some internal edges tend to infinity.
\end{itemize}
As we consider one dimensional moduli spaces, only codimension one degenerations occur, generically. A somewhat counter-intuitive phenomenon can be observed (already appearing for the multiplihedron). In the associahedron, breaking at $k$ edges of a tree is a codimension $k$ phenomenon; while in our setting it can sometimes be of codimension one, as we shall see in more details in Section~\ref{sec:biass_bimult}. Loosely speaking, 
as long as a nonvertical tree is facing on the other side, it forces gluing parameters to be related; while when something vertical is facing, those parameters are unrelated.

This reflects the fact that these spaces are not ``face-coherent'': 
their codimension one boundary faces are not products of $K^{k_i}_{l_i}$'s, but rather \emph{fibered products}. Indeed, $K^k_l$ fibers over $K_k$ and $K_l$, by projecting respectively to the first and second tree. For example, breaking as in Figure~\ref{fig:breaking_K_4_3} 
appearing in the boundary of $K^4_3$ is of codimension one, and corresponds to $( K^2_3 \times_{K_3} K^2_3 )\times K_2$.

\begin{figure}[!h]
    \centering
    \def\svgwidth{.60\textwidth}
\begingroup%
  \makeatletter%
  \providecommand\color[2][]{%
    \errmessage{(Inkscape) Color is used for the text in Inkscape, but the package 'color.sty' is not loaded}%
    \renewcommand\color[2][]{}%
  }%
  \providecommand\transparent[1]{%
    \errmessage{(Inkscape) Transparency is used (non-zero) for the text in Inkscape, but the package 'transparent.sty' is not loaded}%
    \renewcommand\transparent[1]{}%
  }%
  \providecommand\rotatebox[2]{#2}%
  \newcommand*\fsize{\dimexpr\f@size pt\relax}%
  \newcommand*\lineheight[1]{\fontsize{\fsize}{#1\fsize}\selectfont}%
  \ifx\svgwidth\undefined%
    \setlength{\unitlength}{419.52755906bp}%
    \ifx\svgscale\undefined%
      \relax%
    \else%
      \setlength{\unitlength}{\unitlength * \real{\svgscale}}%
    \fi%
  \else%
    \setlength{\unitlength}{\svgwidth}%
  \fi%
  \global\let\svgwidth\undefined%
  \global\let\svgscale\undefined%
  \makeatother%
  \begin{picture}(1,0.33783784)%
    \lineheight{1}%
    \setlength\tabcolsep{0pt}%
    \put(0,0){\includegraphics[width=\unitlength,page=1]{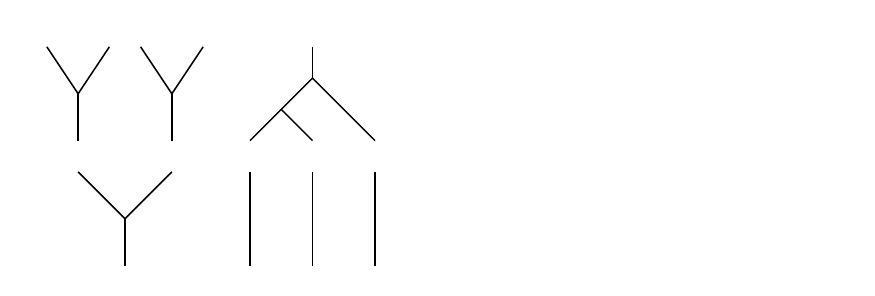}}%
    \put(0.49127263,0.21335931){\color[rgb]{0,0,0}\makebox(0,0)[lt]{\lineheight{1.25}\smash{\begin{tabular}[t]{l}$K^{(2,2)}_{3} \simeq K^{2}_{3}  \times_{K_3} K^{2}_{3}$ \end{tabular}}}}%
    \put(0.48533511,0.07579742){\color[rgb]{0,0,0}\makebox(0,0)[lt]{\lineheight{1.25}\smash{\begin{tabular}[t]{l}$K^{2}_{(1,1,1)} \simeq K_{2}$\end{tabular}}}}%
  \end{picture}%
\endgroup%

      \caption{A fibered product boundary face of $K^4_3$.}
      \label{fig:breaking_K_4_3}
\end{figure}

This ``face-incoherence'' is a priori a problem: the boundary terms in the coherence relations cannot be expressed as combinations of ${\alpha}^k_l$ maps, they are new quantities. 
Possible solutions include: 
\begin{enumerate}
\item ``Continue the moduli spaces'', i.e. glue new moduli spaces to the problematic boundaries, that will have simpler boundaries. This can usually be done by using some resolutions of diagonals, in order to resolve fiber products. 
\item Enlarge the family of moduli spaces, in order to include the fibered products of the boundaries. 
\end{enumerate}

Approach (1) might also work, and perhaps would correspond to \cite{SaneblidzeUmble}, see Section~\ref{ssec:rel_work}. But further subtleties coming from the fact that the diagonals of the associahedra are not associative (see \cite{MarklSchnider_celldiag_ass}) would later lead to complications that we prefer to avoid (not to mention other subtleties of realizing geometrically and coherently the new boundaries).

A priori, approach (2) could never reach an end. But fortunately in our case it does, on a geometrically simple family: trees are replaced by forests. The new moduli spaces have a very natural boundary description, leading to fairly simple coherence relations. Therefore, we follow the second approach (2), as it  seems best suited for the future constructions we have in mind, see Section~\ref{ssec:rel_work}.

We will call this larger family the $f$-biassociahedron, and denote it $K^{\kund}_{\lund}$, with $\kund = (k_1, \ldots , k_a)$ and $\lund = (l_1, \ldots , l_b)$ encoding respectively  the ascending and descending forests (i.e. the ascending forests has $a$ trees, and the $i$-th tree has $k_i$ leaves).  
For example, the boundary strata of Figure~\ref{fig:breaking_K_4_3} will correspond to $K^{(2,2)}_{3} \times K^{2}_{(1,1,1)} $. These will lead, with $A=CM(G)$, to a family of maps:
\e
{\alpha}^{\kund}_{\lund} \colon A^{\otimes b ( k_1 + \cdots + k_a)} \to A^{\otimes a (l_1 + \cdots + l_b)}
\e
satisfying coherence relations of the form:
\e
\partial ( {\alpha}^{\kund}_{\lund}) = \sum_{ ``\partial^1 K^{\kund}_{\lund}\text{''}} {\alpha}^{\kund^0}_{\lund^0}  \circ {\alpha}^{\kund^1}_{\lund^1} ,
\e
as well as some other relations, saying that for some multi-indices, the ${\alpha}^{\kund}_{\lund}$ can be simplified. 
These will be our prototype of an $f$-bialgebra. From $K^{\kund}_{\lund}$ we will also derive several notions of bimodules (and ``quadrimodules''), as well as some categorical analogs. 

Recall that \Ainf -algebras (resp. \Ainf -categories) come with a notion of morphisms (resp. functors), whose coherence relations are encoded by the multiplihedron. 
We also introduce an $f$-multiplihedron $J^{\kund}_{\lund}$ that will permit to define notions of morphisms between $f$-bialgebras and bimodules, and functors between $f$-categories. These will also correspond to moduli spaces of biforests, endowed with a horizontal ``grafting level''. In fact, we will first define $J^{\kund}_{\lund}$, and construct $K^{\kund}_{\lund}$ as a quotient of it.

\subsection{Related work}
\label{ssec:rel_work}
\paragraph*{\textit{Lie Group actions on Fukaya categories}}

The question of how a Lie group acts on a Fukaya category was raised by Teleman in \cite[Conj.~2.9]{Teleman_icm}, he conjectured that ``$G$ should act topologically on $\Fuk(M)$'', and outlined several possible ways of making this statement precise. Progress was made in \cite{EvansLekili}, and a version of this conjecture has been established in  \cite{OhTanaka_quotients_localizations,OhTanaka_actions}. 
Our work can be seen as an alternative approach to this question, and it would be interesting to relate both approaches.
\vspace*{.3cm}
\paragraph*{\textit{Diagonals, biassociahedron and \Ainf -bialgebras}}
Versions of homotopy bialgebras were first defined in \cite{Markl_resolution}, and later revisited in the general setting of properads \cite{MerkulovVallette1,MerkulovVallette2}.  
\Ainf-bialgebras have been introduced in \cite{SaneblidzeUmble}, derived from the combinatorics of their biassociahedron $KK_{m,n}$. 
Markl revisited their construction as a two step process \cite{Markl_biperm_biass}. 
First, he introduces a step-one biassociahedron. 
These are not face-coherent polytopes: their boundaries are not products of themselves, but rather fibered products over the associahedron. He then outlines how one should use a cellular approximation of the  diagonal of the base to resolve the fiber product, leading to a subdivision of the boundaries, and a step-two biassociahedron that should correspond to Saneblidze and Umble's $KK_{m,n}$.  

Cellular  diagonals of the associahedron were first introduced algebraically in \cite{SaneblidzeUmble_diag}, and an alternative construction  appeared in \cite{MarklSchnider_celldiag_ass}. The first topological construction appeared in \cite{MTTVdiag}, where they recover the formula of \cite{MarklSchnider_celldiag_ass}. Both \cite{SaneblidzeUmble_comparing_diag} and \cite{MarklSchnider_celldiag_ass} were recently shown to coincide in \cite{SaneblidzeUmble_comparing_diag} and \cite{DLPS_diag_perm}. 
Such diagonals can  also be used to form tensor products of \Ainf-algebras and categories \cite{Amorim_tensorprod,LAM_diag_multiplihedra,LOT_diag}.  We refer to \cite{DLPS_diag_perm} for more about diagonals, and explicit comparisons between the various ones. There it is shown that there exists exactly two cellular ``operadic'' diagonals of the associahedron and multiplihedron (and more generally, operahedron). Unfortunately, such diagonally-induced tensor products can never be strictly associative \cite{MarklSchnider_celldiag_ass}, and tensor products of \Ainf-morphisms are never strictly compatible with composition \cite{LAM_diag_multiplihedra}. Avoiding the use of diagonals as we do should allow one to overcome these issues.

The step-one biassociahedra have been realized as polytopes in \cite{chapoton2022shuffles}, using a new operation on polytopes called shuffle product. We believe these should correspond to our $K^{\kund}_{\lund}$, with $\kund = (k)$ and $\lund = (l)$ single integers.  
To our best understanding, the step-two biassociahedron have not been realized as polytopes yet. While crucial in combinatorics,  polytope realizations appear to be of secondary interest for our purposes: what matters most is to have coherent geometric realizations, i.e. as moduli spaces of geometric objects one can count in Morse and Floer theory (as in \cite{MauWoodwardrealiz,Botman_witch_realize}). We provide such realizations in Section~\ref{ssec:geom_realiz}, and we believe that Chapoton-Pilaud shuffle product should allow one to realize $K^{\kund}_{\lund}$ as polytopes, for general multi-indices.

\vspace*{.3cm}
\paragraph*{\textit{Koszul duality for props}}
\Ainf -algebras are strong homotopy counterparts of associative algebras: these are algebras over the operad \Ainf , which is the cobar construction of the cooperad Koszul dual to the operad $\Acal ss$ encoding associative algebras. 
Bialgebras are ``gebras'' over a properad $\Bcal iass$ and there exists a Koszul duality theory for properads \cite{vallette_koszul_duality}. Unfortunately the most economical presentation of the properad $\Bcal iass$ is not a homogeneous quadratic so it does not fit directly into the Koszul duality for properads. Therefore the homotopical generalisation of associative bialgebras is a notion much more subtle than \Ainf -algebras and this theory requires further developments.

\vspace*{.3cm}
\paragraph*{\textit{\WW\ theory}} We now explain our main geometric motivations for this work. In \cite{WWfft}\cite{WWffttangles}\cite{Wehrheimphilo}, \WW\ introduce a framework to construct an extended TFT in dimensions $(2..4)$ corresponding to Donaldson-Floer instanton gauge theory \cite{Donaldsonpoly,Floerinst}. Their construction takes the form of a 2-functor from a cobordism (or tangle) 2-category to a 2-category enhancement of Weinstein's symplectic category $\Symp$ \cite{weinsteincat}, in which Lagrangian Floer cohomology classes are viewed as 2-morphisms.

In order to construct a chain-level version of $\Symp$, Bottman introduced a 2-associahedron \cite{Bottman_2_ass} in order to turn  $\Symp$ to an $(A_\infty, 2)$-category \cite{BottmanCarmeli} that contains Fukaya categories as 1-morphism categories.
 He encountered similar issues of polytopes not being ``face-coherent'' as explained before, and suggested in \cite{Bottman_short_note} a strategy to resolve it, which partly inspired our work. A more detailed account will appear in \cite{Bottman_flowcat}. Frédéric Chapoton informed us that he also considered using similar multi-index polytopes back in 1998.

In \cite{ham}, the first author outlined how to extend \WW's constructions one dimension down, i.e. as a $(1..4)$-extended TFT. 
It overcomes issues of moduli spaces having quotient singularities by working equivariantly, following \MW's work \cite{MW}. 
There, the target is (the completion of) a ``partial\footnote{in the sense that some compositions are partially defined} 2-category'' $\Ham$: a real analog of the Moore-Tachikawa category \cite{MooreTachikawa}, in which objects are compact Lie groups, and which contains $\Symp$ as a subcategory of $\mathrm{End}_{\Ham}(\lbrace e \rbrace,\lbrace e \rbrace )$. 
We expect:
\begin{itemize}
\item this work to allow one to promote $\Ham$ to some sort of 3-category (that could be called a partial $(\infty,3)$-category), in the same way that $\Symp$ should be an $(A_\infty, 2)$-category,
\item $f$-bialgebras (or $d$-categories) to  form the objects of (some sort of) 3-category $\fBialg$ (or $\dCat$),
\item some sort of 3-functor $\Ham \to \dCat$ to exist, essentially from Conjecture~\ref{conj:Floer}, extending $\Symp \to A_{\infty}\text{-}\mathcal{C}at$  \cite{MauWehrheimWoodward,Fukaya_functor}.
\end{itemize}
After composing with the functor $\Cob_{1..3}\to \Ham$ of \cite{ham} and extending it to dimension 4, one should get a more algebraic TFT $\Cob_{1..4}\to \dCat$ that could be seen as an instanton counterpart of cornered Heegaard Floer theory, presented below. 

\vspace*{.3cm}

\paragraph*{\textit{\OSz\ theory}}
Donaldson-Floer theory has a conjecturally equivalent cousin: Seiberg-Witten theory \cite{SW_electicmagn_duality}. In the same way that \WW's theory is the Atiyah-Floer counterpart of Donaldson-Floer theory, Heegaard-Floer homology \cite{OSzholodisk1} 
is the symplectic counterpart of Monopole Floer homology \cite{KMbook}. These are TFT-like invariants that associate homology groups to closed 3-manifolds, an one can also go down in dimensions: 
\begin{itemize}
\item a ``bordered'' version of it (allowing for 3-manifolds with boundary) is developped in \cite{LOT_bordered_HF},
\item a ``cornered'' version of it (allowing for 3-manifolds with corners) is developped in \cite{DouglasManolescu,DouglasLipshitzManolescu}, and revisited by \cite{ManionRouquier}, using higher representations of the quantum group (a bialgebra!) $U_q(\mathfrak{gl}(1|1)^+)$, in place of the nilCoxeter 2-algebras involved in \cite{DouglasManolescu,DouglasLipshitzManolescu}. Structural similarities bring hope to bridge these with our constructions in the future.
\end{itemize}

\subsection{Organization of the paper}
\label{ssec:organization}

In Section~\ref{sec:biass_bimult} we introduce the various notions of forests that will be involved in the constructions of the moduli spaces $K^{\kund}_{\lund}$  and  $J^{\kund}_{\lund}$. We orient these spaces, and construct their partial compactifications by gluing moduli spaces of broken biforests, and determine how orientations are affected by these gluings.

This then allows us in Section~\ref{sec:bialg_bimod_mph} to define the various algebraic notions appearing in Conjectures~\ref{conj:Morse}, \ref{conj:Floer} and \ref{conj:wrapped}.

\subsection{Conventions about signs and orientations} 
\label{ssec:orientations}

We hope to provide a clear and unified treatment for signs.  
Most signs will have a simple geometric definition, as the orientation of a map between moduli spaces; and then  explicit formulas will be provided. We will follow the standard conventions below:

\begin{itemize}
\item We orient direct sums as usual: if $(e_1, \ldots, e_m)$ is a direct frame for $E$, and $(f_1, \ldots, f_n)$ is a direct frame for $F$, then $(e_1, \ldots, e_m,f_1, \ldots, f_n)$ is a direct frame for $E\oplus F$.
\item If $E\subset F$ are both oriented, we orient a complement $E^{\perp}$ of $E$ consistently with the splitting $F = E \oplus E^{\perp}$
\item If a connected oriented Lie group $G$ acts freely and properly on an oriented manifold $X$, then the quotient $X/G$ is locally a complement of the $G$-orbits, i.e. $T_x X \simeq T_x (G.x) \oplus T_{[x]} X/G$. We orient it as such.
\item We orient boundaries by the outward pointing first convention.
\end{itemize}
For graded chain complexes:
\begin{itemize}
\item We use Koszul's rule, whenever we exchange tensors $A\otimes B \to B\otimes A$ we use the sign $a\otimes b \mapsto (-1)^{\abs{a}\abs{b}} b\otimes a$. Likewise, if $\varphi , \psi $ are linear maps, $\varphi \otimes \psi$ is defined by $(\varphi \otimes \psi )(x\otimes y)= \varphi(x) \otimes \psi(y) $.

\item If $(A, \partial_A)$,  $(B, \partial_B)$ are chain complexes, $A\otimes B$ is equipped with $id\otimes \partial_B + \partial_A \otimes id$, and $\hom (A,B)$ with $\partial(\varphi) = \partial_{B}\circ \varphi - (-1)^{\abs{\varphi}} \varphi \circ \partial_A$.
\end{itemize}

\begin{acknow}We are indebted to André Henriques the key observation of Figure~\ref{fig:inters_graphs}. We are also grateful to Bruno Vallette for his interest, encouragements, and many insightful conversations and suggestions; as well as to Guillaume Laplante-Anfossi for a careful reading of a preliminary version, and very useful comments and improvements. 

GC would like to thank Nathaniel Bottman for explaining him his work in detail, and the problem of fiber products in the boundary of the 2-associahedron; as well as Frédéric Chapoton, Chris Douglas, Fabian Haiden, Dominic Joyce, Kobi Kremnitzer, Andy Manion, Ciprian Manolescu, Dasha Polyakova, Alex Ritter, Raphaël Rouquier, Constantin Teleman and Chris Woodward for helpful conversations.
\end{acknow}

\section{$f$-Biassociahedron and $f$-Bimultiplihedron}
\label{sec:biass_bimult}

\subsection{Trees, forests and the interior of the $f$-bimultiplihedron} 
\label{ssec:trees_forests}
Here we introduce the kinds of trees and forests that will be involved in our moduli spaces.

\begin{defi}\label{def:forest} A (rooted ribbon) \emph{forest} $\varphi
$ consists in:
\begin{itemize}
\item Ordered finite sets of leaves $\Leaves(\varphi)$ and roots $\Roots(\varphi)$,
\item Finite sets of edges $\Edges(\varphi)$ and vertices  $\Vert(\varphi)$,
\item Source and target maps 
\ea
s &\colon \Edges(\varphi) \to \Vert(\varphi) \cup \Leaves(\varphi),\\
t &\colon \Edges(\varphi) \to \Vert(\varphi) \cup \Roots(\varphi).
\ea
\end{itemize}
Such that:
\begin{itemize}
\item Each leaf (resp. root) is the source (resp. target) of exactly one edge,
\item Each vertex is the source of exactly one edge, and the target of at least two edges,
\item There are no cycles (i.e. cyclic sequences of edges $e_1, \ldots , e_k$ with $t(e_i) = s(e_{i+1})$ and $t(e_k) = s(e_{1})$ are not allowed).

For a vertex $v$, we will denote $\mathrm{In}(v)$ and $\mathrm{Out}(v)$ the sets of incoming and outgoing edges. 
$\mathrm{Out}(v)$ consists in one element. The arity of $v$ is the number of incoming edges $ar(v) = \# \mathrm{In}(v) \geq 2 $, and the  valency $val(v) = 1 + ar(v)$ the number of adjacent edges. We say that a forest is trivalent if 
$val(v)= 3$ for each $v$.
\item The following order-preserving condition is satisfied. The order on $\mathrm{Leaves}(\varphi)$ defines an order on each $\mathrm{In}(v)$ in the following way: if $\gamma$ and $\gamma'$ are two paths going from leaves $l,l'$ to $e,e'\in \mathrm{In}(v)$, then we order $e$ and $e'$ the same way as $l,l'$. This means that we want this order to be independent on the choices of the ``ancestors'' $l,l'$ of $e,e'$). This condition is not satisfied by the right picture in Figure~\ref{fig:forest}.
\end{itemize}

Let the set of \emph{internal edges} be defined as
\e
\IntEd(\varphi) =  s^{-1}(\Vert(\varphi)) \cap t^{-1}(\Vert(\varphi))  \subset \Edges(\varphi).
\e
The remaining edges as refered to as \emph{semi-infinite}, or \emph{infinite} (in the case when the source is a leaf, and the target is a root).

Starting from a leaf $l$, we can form a sequence of edges $e_1, e_2, \ldots $ with $t(e_i) = s(e_{i+1})$. Since there are no cycles, the sequence must terminate to a root $r$. We say that $l$ lies on top of $r$.
 This defines a surjection:
\e
\rho \colon \Leaves(\varphi) \to \Roots(\varphi).
\e
If as an ordered set, $\Roots(\varphi) = (r_1, \ldots , r_a)$, we associate the multi-index of cardinalities:
\e
\kund = (\# \rho^{-1}(r_1), \ldots , \# \rho^{-1}(r_a) )
\e
called the \emph{type} of $\varphi$.

We say that $\varphi$ is a (rooted ribbon) \emph{tree} if it has only one root. So a forest is an ordered union of trees.

The following two kinds of forests will play a particular role. 
We say that $\varphi$ is \emph{vertical} (or trivial) if it has no vertices, i.e. $\kund = (1, \ldots , 1 )$, i.e. all edges are infinite. 
We say that $\varphi$ is \emph{almost vertical} if it is not vertical, is trivalent and has no internal edges, i.e. $\kund$ has only 1 and 2 as coefficients (and at least one 2).
\end{defi}

\begin{figure}[!h]
    \centering
    \def\svgwidth{.50\textwidth}
    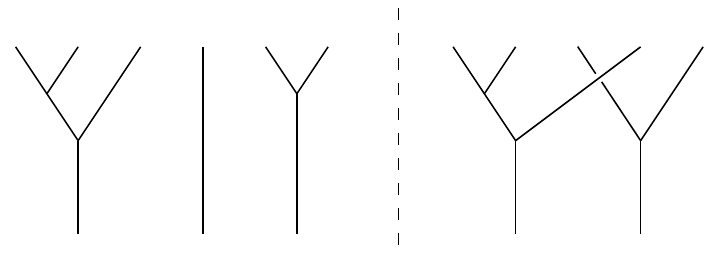
      \caption{The left side corresponds to a forest with $\kund = (3,1,2)$. The right side is not a forest.}
      \label{fig:forest}
\end{figure}

\begin{defi}\label{def:metric_forest} A \emph{metric forest} $F = (\varphi, L)$ consists in a forest $\varphi$, with a length function on the set of internal edges:
\e
L\colon \mathrm{IntEd}(\varphi) \to [0, +\infty)
\e

\end{defi}

We will regard metric forests modulo equivalence.

\begin{defi}\label{def:edge_collapse}
Define the equivalence relation on metric forests generated by either:
\begin{itemize}
\item If $\varphi\simeq \varphi'$\footnote{in the sense that the sets of leaves, roots, edges and vertices are in bijection, consistently with orders, sources and targets.} and under this identification $L\simeq L'$, then we declare $(\varphi, L) \sim (\varphi', L')$.
\item If  $e\in \mathrm{IntEd}(\varphi)$, let $\varphi(e)$ consist in the forest obtained from $\varphi$ by collapsing $e$ (i.e. by removing $e$ and merging $s(e)$ 	and $t(e)$). If $L(e) = 0$, then we declare $F=(\varphi,L)$ to be equivalent to $F' = (\varphi(e), L')$, with $L'=L$ on $\mathrm{IntEd}(\varphi') =  \mathrm{IntEd}(\varphi)\setminus \lbrace e\rbrace$.
\end{itemize}
We say that $F$ is \emph{irreducible} if $L(e) > 0$ for every edge. Each metric forest can be put in a unique irreducible form.
\end{defi}

In this paper we will be interested in metric forests with extra parameters (one height for each nonvertical tree)

\begin{defi}[Ascending and descending forests]\label{def:up_down_forest}An \emph{ascending forest} $U = (\varphi,h)$ is a forest $\varphi$ with a height function $h\colon \Vert(\varphi) \to \rr$ such that if $e$ is an internal edge from $v_1$ to $v_2$, then $h(v_1)\geq h(v_2)$. That is, edges are pointing down. For leaves $l$ and roots $r$, we set $h(l)=+\infty$ and $h(r)= -\infty$.

Ascending forests are regarded up to equivalence as well: if for some edge $e\colon v_1\to v_2$, one has $h(v_1)= h(v_2)$, then one identifies $U$ with $(\varphi(e), h_e)$, with $h_e$ assigning the common value to the merged vertex $v_e$. 
Up to this equivalence they form a moduli space $U^{\kund}$, for a given type $\kund$. 

An ascending forest $U = (\varphi,h)$ determines a metric forest $T = (\varphi,L)$ by setting $L(e) = h(v_1) - h(v_2)$ for $e\colon v_1\to v_2$.

Let $\tilde{a}$ stand for the number of nonvertical trees in $\varphi$. The group $\rr^{\tilde{a}}$ acts on $U^{\kund}$ by shifting the heights of the nonvertical trees. This action is free, and leaves the underlying metric forest unchanged.

Likewise, one can define a \emph{descending forest} $D = (\varphi,h)$ to be such that $(\varphi,-h)$ is an ascending forest. For descending forests we will use multi-indices $\lund = (l_1, \ldots , l_b)$, and denote their moduli spaces $D_{\lund}$, acted on by $\rr^{\tilde{b}}$.

\end{defi}

\begin{defi}\label{def:up_down_forests}

A \emph{biforest} is a pair  $(U,D)$ of an ascending  and a descending forest. For a pair of multi-indices $\kund = (k_1, \ldots , k_a)$  and $\lund = (l_1, \ldots , l_b)$, let the (interior of the) bi-multiplihedron be defined as
\e
J^{\kund}_{\lund} = U^{\kund} \times D_{\lund}.
\e
\end{defi}
Geometrically, we think of elements in $J^{\kund}_{\lund}$ as having a ``grafting level'' at height $h=0$, as in Figure~\ref{fig:1_grafted_biforest}.

\begin{figure}[!h]
    \centering
    \def\svgwidth{.50\textwidth}
\begingroup%
  \makeatletter%
  \providecommand\color[2][]{%
    \errmessage{(Inkscape) Color is used for the text in Inkscape, but the package 'color.sty' is not loaded}%
    \renewcommand\color[2][]{}%
  }%
  \providecommand\transparent[1]{%
    \errmessage{(Inkscape) Transparency is used (non-zero) for the text in Inkscape, but the package 'transparent.sty' is not loaded}%
    \renewcommand\transparent[1]{}%
  }%
  \providecommand\rotatebox[2]{#2}%
  \newcommand*\fsize{\dimexpr\f@size pt\relax}%
  \newcommand*\lineheight[1]{\fontsize{\fsize}{#1\fsize}\selectfont}%
  \ifx\svgwidth\undefined%
    \setlength{\unitlength}{368.50393701bp}%
    \ifx\svgscale\undefined%
      \relax%
    \else%
      \setlength{\unitlength}{\unitlength * \real{\svgscale}}%
    \fi%
  \else%
    \setlength{\unitlength}{\svgwidth}%
  \fi%
  \global\let\svgwidth\undefined%
  \global\let\svgscale\undefined%
  \makeatother%
  \begin{picture}(1,0.44615385)%
    \lineheight{1}%
    \setlength\tabcolsep{0pt}%
    \put(0,0){\includegraphics[width=\unitlength,page=1]{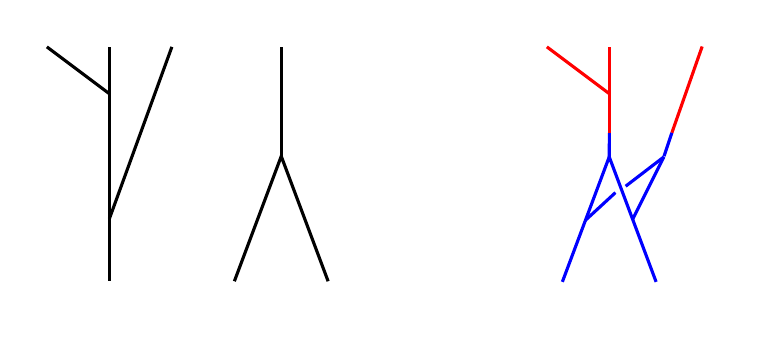}}%
    \put(0.58004808,0.25323796){\color[rgb]{0,0,0}\makebox(0,0)[lt]{\lineheight{1.25}\smash{\begin{tabular}[t]{l}$0$\end{tabular}}}}%
    \put(0,0){\includegraphics[width=\unitlength,page=2]{1_grafted_biforest.pdf}}%
  \end{picture}%
\endgroup%

      \caption{Metric biforests have a grafting level at height $h=0$.}
      \label{fig:1_grafted_biforest}
\end{figure}

We will use the following notations associated to a multi-index $\kund$:
\begin{itemize}
\item $\abs{\kund} = k_1 + \cdots + k_a$ corresponds to the number of leaves of forests
\item $n({\kund})=a$ the number of trees (i.e. the number of roots)
\item $\tilde{\kund}$ the multi-index obtained by removing all the entries in $\kund$ equal to 1.
\item $\tilde{a} = \tilde{n}({\kund}) = n(\tilde{\kund})$ the number of nonvertical trees
\item $v(\kund) = \abs{\kund} - n({\kund})$ the generic number of vertices (see below).
\item we will denote $\One_a= (1, \ldots,1)$ the multi-index with $a$ entries equal to 1. 
\end{itemize}

The spaces $U^{\kund}$, $D_{\lund}$ and $J^{\kund}_{\lund}$ can be explicitly described, and oriented, as follows. For a multi-index $\kund$, define the ordered set
\e
\Vert(\kund) = \left\lbrace 1, \ldots , \abs{\kund} \right\rbrace \setminus \left\lbrace k_1, k_1 + k_2, \ldots , \abs{\kund} \right\rbrace ;
\e
which has cardinality $v(\kund) = \abs{\kund} - n({\kund})$. Observe that if $\varphi$ is a forest of type $\kund$, one has a surjective map $\Vert(\kund)\to\Vert(\varphi)$ which is bijective iff $\varphi$ is trivalent. Indeed, the elements of $\Vert(\kund)$ correspond to pairs of consecutive leaves $l,l'$ lying on the same tree (i.e. over the same root $r$). Map such a pair to the first vertex where the paths going from $l,l'$ to $r$ meet.

It is straightforward to check that the map it induces
\e\label{eq:identif_bimult}
i^{\kund}_{\lund} \colon J^{\kund}_{\lund} \to \rr^{\Vert(\kund) \cup \Vert(\lund)} \simeq \rr^{v(\kund)} \oplus  \rr^{v(\lund)},
\e
mapping a biforest to all its heights (first the ascending forest, then the descending one); is a homeomorphism. We use it to give $J^{\kund}_{\lund}$ the structure of an oriented vector space. This will be useful for computing orientations and signs.

\subsection{Geometric realizations}
\label{ssec:geom_realiz}

The (interior of the) $f$-biassociahedron $K^{\kund}_{\lund}$ will be defined as a quotient of $J^{\kund}_{\lund}$, corresponding to removing the grafting level. Given a biforest $(U,D)$, one can form a metric graph (explained below), and we will quotient by the symmetries leaving this graph unchanged.

If  $U = (\varphi, h)$ is either an ascending or descending forest we associate the following spaces

\ea
\abs{U} &= \coprod_{e\in \Edges(\varphi)}I_e, \\
\langle U \rangle &= \abs{U}_{/_\sim}
\ea
where $I_e \subset \rr$ is either $[h(v_2), h(v_1)]$ or  $[h(v_1), h(v_2)]$, if $e\colon v_1\to v_2$ (when $h(v_i) = \pm \infty$ the bound is excluded). In $\langle U \rangle$ the equivalence relation we quotient out by is given by  identifying endpoints corresponding to a given vertex.

We also still denote $h\colon \abs{U} \to \rr$ and $h\colon \langle U \rangle \to \rr$ the obvious ``vertical'' projections. The graph associated to the biforest, which we will refer to as the \emph{Henriques intersection}, is defined to be the fibered product
\e
\Gamma(U, D) = \langle U \rangle\times_{\rr} \langle D \rangle,
\e
with induced height function. This corresponds to the intersection picture mentioned in the introduction (Figure~\ref{fig:inters_graphs}).

\begin{figure}[!h]
    \centering
    \def\svgwidth{.90\textwidth}
    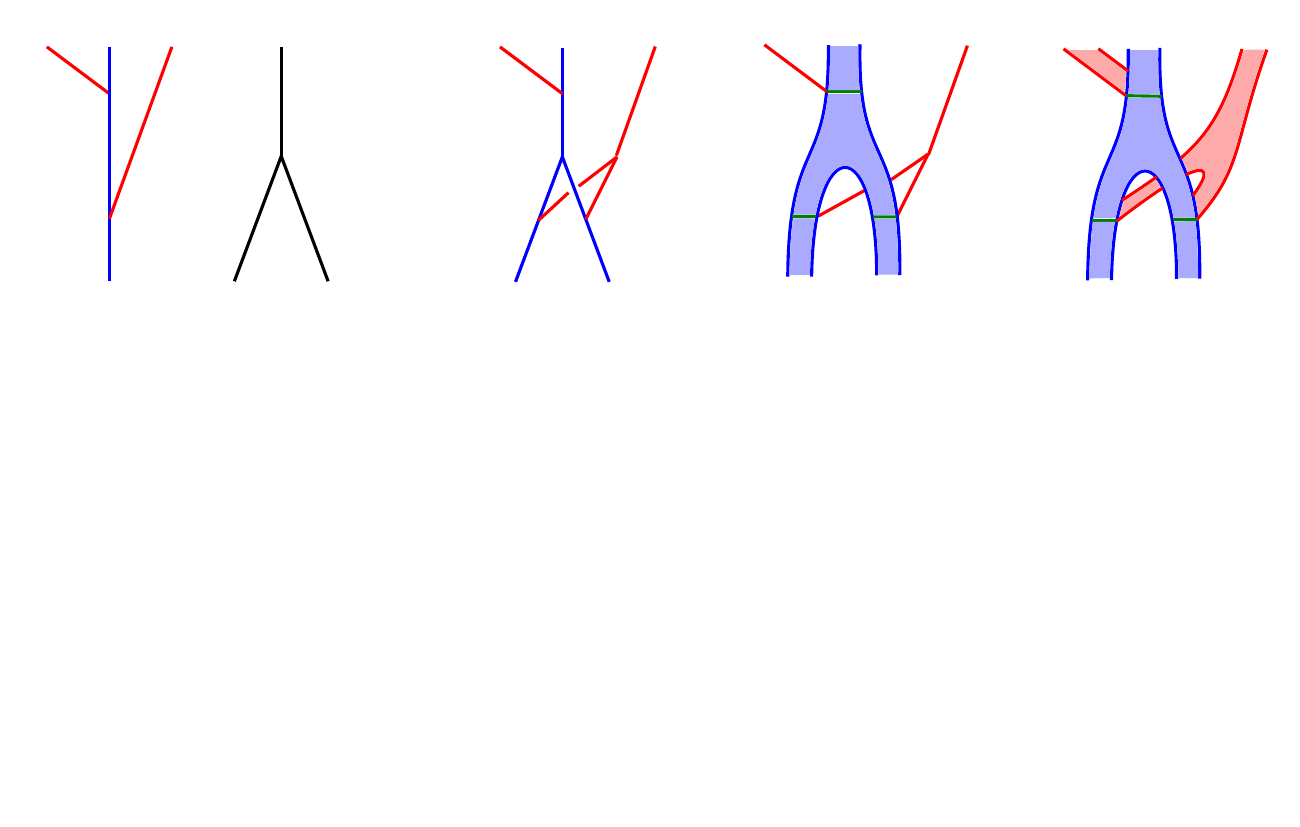
      \caption{A biforest $(U,D)$, the corresponding graph, hybrid graph, and foam, and their closed counterpart. We have drawn $\abs{U}^{\mathrm{Morse}}$ in red (corresponding to the groups), and $\abs{U}^{\mathrm{Floer}}$ in blue (corresponding to the $G\times H$-manifolds). In $H(U, D)$ and $H_c(U, D)$, the vertices of the red edges should ``touch'' every points of the green seam (condition (\ref{eq:hybrid_vert_cond})).}
      \label{fig:graph_hybr_foam}
\end{figure}

For Morse and Floer theory experts,  we describe the domains that can be used in order to form the moduli spaces involved in  Conjectures~\ref{conj:Morse}, \ref{conj:Floer} and \ref{conj:wrapped}, see Figure~\ref{fig:graph_hybr_foam}. 
In Morse theory (Conjecture~\ref{conj:Morse}), to be precise,  we will actually use $\abs{U}\times_{\rr} \abs{D} $ rather than $\Gamma(U, D)$.
In Floer (open string) theory, we will consider a thickened version $\langle\!\langle D \rangle\!\rangle$ of $\langle D \rangle$, a Riemann surface made of \emph{shredded strips}:
\e
\langle\!\langle D \rangle\!\rangle = \bigcup_{i=1}^b \left( \left( [0, l_i] \times\rr \right)\setminus\!\!\!\setminus \left(\bigcup_{i'=1}^{l_i-1} \lbrace i' \rbrace\times (-\infty, h_{i,i'}) \right)\right),
\e
Where the symbol $\setminus\!\!\setminus$ stands for ``cutting'', i.e. removing and gluing back twice. 
For Conjecture~\ref{conj:wrapped}, we will consider \emph{pseudo-holomorphic foams} with domain
\e
F(U, D) =  \abs{U} \times_{\rr} \langle\!\langle D \rangle\!\rangle,
\e
and for Conjecture~\ref{conj:Floer}, we will consider a hybridation of $\Gamma(U, D)$ and $F(U, D)$, as in \cite{equiv_trees}.
Writing $ \abs{U} = \abs{U}^{\mathrm{Morse}}\cup \abs{U}^{\mathrm{Floer}}$ the parts corresponding to the groups $G,H,$ and to the manifold $M$ respectively, the associated \emph{hybrid graph} is:
\e
H(U, D) = \abs{U}^{\mathrm{Morse}} \times_{\rr}  \langle D \rangle \cup \abs{U}^{\mathrm{Floer}}\times_{\rr} \langle\!\langle D \rangle\!\rangle.
\e
As in \cite{PSS,Oh_relat_Floer_QC,CorneaLalonde}, these are mixing flowlines and pseudo-holomorphic curves, but the way they interact is different. The vertex conditions, for triples of maps locally of the form
\ea
\gamma &\colon (-\infty, 0] \to G, \\
u_- &\colon (-\infty, 0] \times [0,1] \to M, \\
u_+ &\colon [0, +\infty) \times [0,1] \to M, 
\ea
will be given by:
\e\label{eq:hybrid_vert_cond}
\forall t\in [0,1],\  \gamma(0)\cdot u_-(0,t) = u_+(0,t) .
\e

If one  considers closed strings instead (quantum or symplectic homology), one might replace $\langle\!\langle D \rangle\!\rangle$ by the punctured sphere:
\e
\lbrace D \rbrace = \langle\!\langle D \rangle\!\rangle^- \cup_{\partial \langle\!\langle D \rangle\!\rangle} \langle\!\langle D \rangle\!\rangle,
\e
and form the closed counterparts of $H(U, D)$ and $F(U, D)$:
\ea
H_c(U, D) &= \abs{U}^{\mathrm{Morse}} \times_{\rr}  \langle D \rangle \cup \abs{U}^{\mathrm{Floer}}\times_{\rr} \lbrace D \rbrace, \\
F_c(U, D) &=  \abs{U} \times_{\rr} \lbrace D \rbrace.
\ea

\subsection{The interior of the $f$-biassociahedron}
\label{ssec:int_f_biass}

The group of symmetries of $J^{\kund}_{\lund}$ preserving $\Gamma(U, D)$ as a metric graph\footnote{i.e. the graph with  lenghts of its edges.} is:
\e
G^{\kund}_{\lund} = \begin{cases} 
\rr^{\tilde{a}} &\text{ if }\lund = \One_b,\\
\rr^{\tilde{b}} &\text{ if }\kund = \One_a ,\\
\rr&\text{ otherwise.}
\end{cases}
\e

This is an oriented vector space, and under the identification (\ref{eq:identif_bimult}), it acts in the first above case by translations along vectors $v_1, \ldots , v_{\tilde{a}}$, with 
\e
v_i = (0, \ldots , 0, 1, \ldots , 1, 0 ,\ldots,  0)
\e 
corresponding to the vertices of the i-th nonvertical tree. Likewise in the second case. In the third case it acts by an overall shift along $v=(1, \ldots, 1)$.

\begin{defi}\label{def:int_biass}
Let then the (interior of the) forest biassociahedron be defined as:
\e
{K}^{\kund}_{\lund} = \begin{cases}J^{\kund}_{\lund} /G^{\kund}_{\lund} &\text{ if }\dim G^{\kund}_{\lund} =1 ,\\ \emptyset & \text{ otherwise.}\end{cases}
\e
We will occasionally consider the variation $\widetilde{K}^{\kund}_{\lund} = J^{\kund}_{\lund} /G^{\kund}_{\lund}$, regardless of $\dim G^{\kund}_{\lund}$. 
We orient both spaces as quotients of $J^{\kund}_{\lund}$. That is, as oriented vector spaces, 
$ J^{\kund}_{\lund} =  G^{\kund}_{\lund} \oplus \widetilde{K}^{\kund}_{\lund}$, and when not empty, $ J^{\kund}_{\lund} =  G^{\kund}_{\lund} \oplus {K}^{\kund}_{\lund}$. In particular,  we identify $K^{\kund}_{\lund}$ with the standard orthogonal complement of $G^{\kund}_{\lund}$ in $J^{\kund}_{\lund} \simeq \rr^{v(\kund) + v(\lund)}$.
\end{defi}

\begin{remark}The variant $\widetilde{K}^{\kund}_{\lund}$ seems more natural to consider, and is convenient for describing the higher codimensional corner strata. However it is redundant: when $\dim G^{\kund}_{\lund}\geq 2$, $\widetilde{K}^{\kund}_{\lund}$ is a product of associahedrons up to a sign (Proposition~\ref{prop:K_vertic_split} below). It turns out to be more convenient to set ${K}^{\kund}_{\lund} = \emptyset$ and the associated maps ${\alpha}^{\kund}_{\lund} = 0$ in the definition of $f$-bialgebras, and this causes no loss of information.
\end{remark}
For a multi-index $\kund$, let us introduce:
\e\label{eq:spadesuit}
{\spadesuit_{\kund}} = \sum_{i=1}^{a} (a-i) k_i.
\e

\begin{prop}\label{prop:K_vertic_split}
Assume that $\lund$  is vertical, then as oriented manifolds:
\e\label{eq:K_lund_vertic}
\widetilde{K}^{\kund}_{\lund} \simeq (-1)^{\spadesuit_{\tilde{\kund}}} K_{k_1} \times \cdots \times K_{k_a}.
\e
Notice the tilde in $\spadesuit_{\tilde{\kund}}$, with $\tilde{\kund}$ standing for $\kund$ with the vertical trees removed.

Likewise, if $\kund$  is vertical, then as oriented manifolds:
\e\label{eq:K_kund_vertic}
\widetilde{K}^{\kund}_{\lund} \simeq (-1)^{\spadesuit_{\tilde{\lund}}} K_{l_1} \times \cdots \times K_{l_b}.
\e
\end{prop}

\begin{proof}
Observe first that the obvious $\widetilde{K}^{\kund}_{\lund} \simeq \widetilde{K}^{\tilde{\kund}}_{\tilde{\lund}} $ is orientation preserving, since it commutes with the identifications $i^{\kund}_{\lund}$. Assume therefore that all vertical trees have been removed, so that $\tilde{\kund} = \kund$. On the one hand, $ K_{k_1} \times \cdots \times K_{k_a}$ is oriented so that
\e
J^{\kund}_{\lund} \simeq ( \langle v_1 \rangle \oplus K_{k_1}) \oplus \cdots \oplus (\langle v_{a} \rangle \oplus K_{k_a}).
\e
On the other hand, $\widetilde{K}^{\kund}_{\lund}$ is oriented so that
\e
J^{\kund}_{\lund} \simeq ( \langle v_1 \rangle \oplus \cdots \oplus \langle v_{a} \rangle ) \oplus ( K_{k_1} \oplus \cdots \oplus K_{k_a}).
\e
The sign ${\spadesuit_{\kund}}$ corresponds to moving all the $v_i$'s upfront: first exchanging $v_2$ with $K_{k_1}$ gives $\dim K_{k_1} = k_1\ mod\ 2$, then exchanging $v_3$ with  $K_{k_1} \oplus K_{k_2}$ gives $k_1+ k_2$, etc until $k_1 + \cdots + k_{a-1}$.
\end{proof}

The following is immediate from the definitions:
\begin{prop}\label{prop:dim_K_J}The dimensions of $J^{\kund}_{\lund}$, $K^{\kund}_{\lund}$ and $\widetilde{K}^{\kund}_{\lund}$ are given by
\ea
\dim J^{\kund}_{\lund} &= v({\kund}) + v({\lund}), \\
\dim K^{\kund}_{\lund} &= v({\kund}) + v({\lund}) -1, \\
\dim \widetilde{K}^{\kund}_{\lund} &= v({\kund}) + v({\lund}) -c(\kund,\lund), 
\ea
with
\e
c(\kund,\lund) = \dim G^{\kund}_{\lund} = \begin{cases} 
\tilde{a} &\text{ if }\tilde{b}=0,\\
\tilde{b} &\text{ if }\tilde{a}=0,\\
1&\text{ otherwise.}
\end{cases}
\e
\end{prop}
\arnaque

\subsection{Partial compactification}
\label{ssec:compactification}

Topologically, ${K}^{\kund}_{\lund}$ and $J^{\kund}_{\lund}$ are just vector spaces, what is interesting is their compactification, and the combinatorics of their boundary. 
They can be compactified to compact manifolds with generalized corners (as in Joyce \cite{Joyce_corners}) $\overline{K}^{\kund}_{\lund}$ and $\overline{J}^{\kund}_{\lund}$, by gluing to them moduli spaces of ``broken biforests'' of various codimensions. 
We expect these spaces to have polytope realizations, as in  \cite{chapoton2022shuffles}, using ``shuffle products''.

As explained in the introduction, what we mostly care about in Morse and Floer theory are the ``partial compactifications'' $\left(\overline{K}^{\kund}_{\lund}\right)_{\leq 1}$ and $\left(\overline{J}^{\kund}_{\lund} \right)_{\leq 1}$ obtained by only adding the codimension one boundary strata. 
We construct these in detail, and quickly outline how one could construct the full compactifications. Polytope constructions might be more direct for doing so, though our constructions are more transparently related to Morse/Floer theory applications.

The important feature of these spaces is that they are ``face-coherent'', i.e. their boundaries are products of themselves. This is what allows one to write down coherent relations for the maps they will define, leading to the notions in Section~\ref{sec:bialg_bimod_mph}.

One can ``glue'' multi-indices as in Figure~\ref{fig:glueing_multind}:

\begin{figure}[!h]
    \centering
    \def\svgwidth{.50\textwidth}
\begingroup%
  \makeatletter%
  \providecommand\color[2][]{%
    \errmessage{(Inkscape) Color is used for the text in Inkscape, but the package 'color.sty' is not loaded}%
    \renewcommand\color[2][]{}%
  }%
  \providecommand\transparent[1]{%
    \errmessage{(Inkscape) Transparency is used (non-zero) for the text in Inkscape, but the package 'transparent.sty' is not loaded}%
    \renewcommand\transparent[1]{}%
  }%
  \providecommand\rotatebox[2]{#2}%
  \newcommand*\fsize{\dimexpr\f@size pt\relax}%
  \newcommand*\lineheight[1]{\fontsize{\fsize}{#1\fsize}\selectfont}%
  \ifx\svgwidth\undefined%
    \setlength{\unitlength}{229.60629921bp}%
    \ifx\svgscale\undefined%
      \relax%
    \else%
      \setlength{\unitlength}{\unitlength * \real{\svgscale}}%
    \fi%
  \else%
    \setlength{\unitlength}{\svgwidth}%
  \fi%
  \global\let\svgwidth\undefined%
  \global\let\svgscale\undefined%
  \makeatother%
  \begin{picture}(1,0.51851852)%
    \lineheight{1}%
    \setlength\tabcolsep{0pt}%
    \put(0,0){\includegraphics[width=\unitlength,page=1]{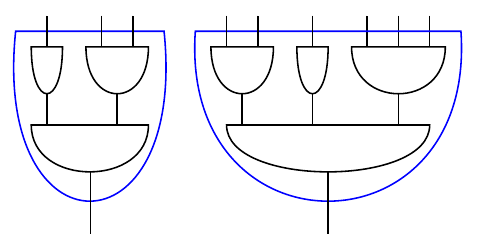}}%
    \put(0.02911351,0.09464317){\color[rgb]{0,0,1}\makebox(0,0)[lt]{\lineheight{1.25}\smash{\begin{tabular}[t]{l}$\kund$\end{tabular}}}}%
    \put(-0.05292602,0.19780139){\color[rgb]{0,0,0}\makebox(0,0)[lt]{\lineheight{1.25}\smash{\begin{tabular}[t]{l}$\kund^0$\end{tabular}}}}%
    \put(-0.05376885,0.35810106){\color[rgb]{0,0,0}\makebox(0,0)[lt]{\lineheight{1.25}\smash{\begin{tabular}[t]{l}$\kund^1$\end{tabular}}}}%
  \end{picture}%
\endgroup%

      \caption{Gluing $\kund^1=(1,2,2,1,3)$ on top of $\kund^0=(2,3)$ gives $\kund = \kund^1 \sharp \kund^0 = (3,6)$.}
      \label{fig:glueing_multind}
\end{figure}

\begin{defi}\label{def:glueing_mult_ind} Let $\kund^0 , \kund^1$ be two multi-indices such that $\abs{\kund^0} = n(\kund^1)$. Let ``\emph{$\kund^1$ glued on top of $\kund^0$}'' be the multi-index defined as:
\e
\kund^1 \sharp \kund^0 = ( k^1_1 + \cdots + k^1_{k^0_1} , k^1_{k^0_1 +1} + \cdots + k^1_{k^0_1 + k^0_2}, \cdots , k^1_{k^0_1 +\cdots + k^0_{a^0-1} +1} + \cdots + k^1_{a^1}).
\e
\end{defi}
We have $\abs{\kund^1 \sharp \kund^0} = \abs{\kund^1 } $, $n(\kund^1 \sharp \kund^0) = n(\kund^0) $, and this operation is associative. Moreover, one has a natural identification 
\e\label{eq:split_set_vertices}
\Vert(\kund^1 \sharp \kund^0) \simeq  \Vert(\kund^1 ) \cup \Vert( \kund^0) .
\e
Indeed, one has $\Vert(\kund^1 ) \subset \Vert(\kund^1 \sharp \kund^0)$, and $\Vert(\kund^1 \sharp \kund^0) \setminus \Vert(\kund^1 )$ is in an order-preserving one-to-one correspondence with $\Vert( \kund^0)$.

In particular, given $\kund$, there are $2^{v(\kund)}$ ways to split $\kund$ as $\kund^1 \sharp \kund^0$, which corresponds to the number of ways of splitting  $\Vert(\kund)$ into two subsets.

It follows from (\ref{eq:split_set_vertices}) and the identification (\ref{eq:identif_bimult}) that one has a natural (unoriented, the sign change will be computed in Section~\ref{ssec:orient}) identification 
\e
J^{\kund}_{\lund} \simeq J^{\kund^0}_{\lund^0} \oplus J^{\kund^1}_{\lund^1}.
\e

Denote
\ea 
v^i &= (1, \ldots , 1) \in G^{\kund^i}_{\lund^i} \subset J^{\kund^i}_{\lund^i}\simeq \rr^{v(\kund^i) + v(\lund^i)},\text{ and} \\ 
v &= v^0+ v^1 = (1, \ldots , 1) \in G^{\kund}_{\lund} \subset J^{\kund}_{\lund}. 
\ea

If $c(\kund^1, \lund^1)=1$, then one has $J^{\kund^1}_{\lund^1} \simeq \rr v^1 \oplus K^{\kund^1}_{\lund^1}$. Call the resulting identification the gluing map:
\ea
g^{\kund^0,\kund^1}_{\lund^0,\lund^1; 0} \colon \rr \times J^{\kund^0}_{\lund^0} \times K^{\kund^1}_{\lund^1} &\to J^{\kund}_{\lund} \\ 
(h, B_0, B_1) &\mapsto h v_1  + B_0+ B_1 .
\ea
Observe that when the parameter $h$ is sufficiently large, this corresponds to gluing the biforest of $K^{\kund^1}_{\lund^1}$ (shifted up by $h$) on top of the one of $J^{\kund^0}_{\lund^0}$, as in Figure~\ref{fig:glueing_biforest}.

\begin{figure}[!h]
    \centering
    \def\svgwidth{.50\textwidth}
\begingroup%
  \makeatletter%
  \providecommand\color[2][]{%
    \errmessage{(Inkscape) Color is used for the text in Inkscape, but the package 'color.sty' is not loaded}%
    \renewcommand\color[2][]{}%
  }%
  \providecommand\transparent[1]{%
    \errmessage{(Inkscape) Transparency is used (non-zero) for the text in Inkscape, but the package 'transparent.sty' is not loaded}%
    \renewcommand\transparent[1]{}%
  }%
  \providecommand\rotatebox[2]{#2}%
  \newcommand*\fsize{\dimexpr\f@size pt\relax}%
  \newcommand*\lineheight[1]{\fontsize{\fsize}{#1\fsize}\selectfont}%
  \ifx\svgwidth\undefined%
    \setlength{\unitlength}{229.60629921bp}%
    \ifx\svgscale\undefined%
      \relax%
    \else%
      \setlength{\unitlength}{\unitlength * \real{\svgscale}}%
    \fi%
  \else%
    \setlength{\unitlength}{\svgwidth}%
  \fi%
  \global\let\svgwidth\undefined%
  \global\let\svgscale\undefined%
  \makeatother%
  \begin{picture}(1,0.59259259)%
    \lineheight{1}%
    \setlength\tabcolsep{0pt}%
    \put(0,0){\includegraphics[width=\unitlength,page=1]{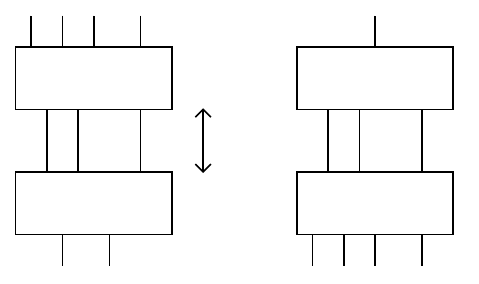}}%
    \put(0.45464035,0.2712962){\color[rgb]{0,0,0}\makebox(0,0)[lt]{\lineheight{1.25}\smash{\begin{tabular}[t]{l}$h$\end{tabular}}}}%
    \put(0.14881549,0.15473598){\color[rgb]{0,0,0}\makebox(0,0)[lt]{\lineheight{1.25}\smash{\begin{tabular}[t]{l}$\Uund^0$\end{tabular}}}}%
    \put(0.1479673,0.41782779){\color[rgb]{0,0,0}\makebox(0,0)[lt]{\lineheight{1.25}\smash{\begin{tabular}[t]{l}$\Uund^1$\end{tabular}}}}%
    \put(0.73687714,0.15216297){\color[rgb]{0,0,0}\makebox(0,0)[lt]{\lineheight{1.25}\smash{\begin{tabular}[t]{l}$\Dund^0$\end{tabular}}}}%
    \put(0.73569902,0.41730185){\color[rgb]{0,0,0}\makebox(0,0)[lt]{\lineheight{1.25}\smash{\begin{tabular}[t]{l}$\Dund^1$\end{tabular}}}}%
  \end{picture}%
\endgroup%

      \caption{Gluing two biforests.}
      \label{fig:glueing_biforest}
\end{figure}

Likewise, if $c(\kund^0, \lund^0)=1$, then one has $J^{\kund^0}_{\lund^0} \simeq \rr v^0 \oplus K^{\kund^0}_{\lund^0}$. Identify $\rr \simeq \rr (-v^0)$, and call the resulting identification the gluing map:
\ea
g^{\kund^0,\kund^1}_{\lund^0,\lund^1; 1} \colon \rr \times K^{\kund^0}_{\lund^0} \times J^{\kund^1}_{\lund^1} &\to J^{\kund}_{\lund}\\ 
(h, B_0, B_1) &\mapsto -h v_0  + B_0+ B_1 .
\ea
This corresponds to gluing the biforest of $K^{\kund^0}_{\lund^0}$, shifted down by $h\in \rr$, below the one of $J^{\kund^1}_{\lund^1}$.

Assume now that $c(\kund, \lund)=c(\kund^0, \lund^0)=c(\kund^1, \lund^1)=1$, then let the gluing map 
\e
g^{\kund^0,\kund^1}_{\lund^0,\lund^1} \colon \rr \times K^{\kund^0}_{\lund^0} \times K^{\kund^1}_{\lund^1} \to K^{\kund}_{\lund}.
\e
be obtained by composing the map 
\ea
\rr \times K^{\kund^0}_{\lund^0} \times K^{\kund^1}_{\lund^1} &\to J^{\kund}_{\lund}, \\ (h, B_0, B_1) &\mapsto h v_1  + B_0+ B_1 ,
\ea
with the projection $J^{\kund}_{\lund}\to K^{\kund}_{\lund}$. 
Then define the partial compactifications:
\ea
 \left( \overline{K}^{\kund}_{\lund} \right)_{\leq 1} &=  \left( {K}^{\kund}_{\lund} \cup \bigcup_{\begin{subarray}{c} \kund = \kund^1 \sharp \kund^0 \\ \lund = \lund^0 \sharp \lund^1 \end{subarray}} (-\infty, + \infty] \times {K}^{\kund^0}_{\lund^0}\times {K}^{\kund^1}_{\lund^1} \right) / \sim \\
 \left( \overline{J}^{\kund}_{\lund} \right)_{\leq 1} &=  \left( {J}^{\kund}_{\lund} \cup \bigcup_{\begin{subarray}{c} \kund = \kund^1 \sharp \kund^0 \\ \lund = \lund^0 \sharp \lund^1 \end{subarray}} (-\infty, + \infty] \times \left( {J}^{\kund^0}_{\lund^0}\times {K}^{\kund^1}_{\lund^1} \cup {K}^{\kund^0}_{\lund^0}\times {J}^{\kund^1}_{\lund^1}  \right) \right) / \sim ,
\ea
Where $\sim$ identifies points with their image under the gluing maps. 
These are manifolds with boundary:

\ea
\partial \left( \overline{K}^{\kund}_{\lund} \right)_{\leq 1} &=   \bigcup_{\begin{subarray}{c} \kund = \kund^1 \sharp \kund^0 \\ \lund = \lund^0 \sharp \lund^1 \end{subarray}} (-1)^{ \rho  } \cdot {K}^{\kund^0}_{\lund^0}\times {K}^{\kund^1}_{\lund^1} , \label{eq:bdry_K}\\
 \partial \left( \overline{J}^{\kund}_{\lund} \right)_{\leq 1} &=  \bigcup_{\begin{subarray}{c} \kund = \kund^1 \sharp \kund^0 \\ \lund = \lund^0 \sharp \lund^1 \end{subarray}} \left( (-1)^{ \rho_0  } \cdot {J}^{\kund^0}_{\lund^0}\times {K}^{\kund^1}_{\lund^1} \cup (-1)^{ \rho_1  } \cdot {K}^{\kund^0}_{\lund^0}\times {J}^{\kund^1}_{\lund^1}  \right) \label{eq:bdry_J},
\ea
where ${ \rho = \rho^{\kund^0,\kund^1}_{\lund^0,\lund^1}  }, { \rho_0 = \rho^{\kund^0,\kund^1}_{\lund^0,\lund^1;0}  }, { \rho_1 = \rho^{\kund^0,\kund^1}_{\lund^0,\lund^1;1}  } \in \Z{2}$ stand for the orientations of the gluing maps ${ g^{\kund^0,\kund^1}_{\lund^0,\lund^1}  }, { g^{\kund^0,\kund^1}_{\lund^0,\lund^1;0}  }, { g^{\kund^0,\kund^1}_{\lund^0,\lund^1;1}  } $, which will appear in the coherence relations of  $f$-bialgebras 
(\ref{eq:R_kund_lund})  (resp. of morphisms of $f$-algebras (\ref{eq:M_kund_lund_mph})); and which we are going to determine now.

\subsection{Orientations and signs}
\label{ssec:orient}

Given a multi-index $\kund = (k_1, \ldots , k_a)$, leaves of the corresponding forest can be indexed in two different ways. Either by a pair of integers $(i,i')$, with $1\leq i \leq a$ and $1\leq i' \leq k_i$, corresponding to the $i'$-th leaf of the $i$-th tree; or by a single integer $h$ such that $1\leq h \leq \abs{\kund}$, corresponding to the position in the overall set of leaves. These are related by $h= s_{\kund}(i,i')$, with:
\e\label{eq:s_kund}
s_{\kund}(i,i') = k_1 + \cdots + k_{i-1} + i'.
\e
If $h= s_{\kund}(i,i')$, let $v_{\geq h}(\kund)$ stand for ``the number of vertices at the right side of $h$'':
\e
v_{\geq h}(\kund) = \mathrm{Card} \left(\Vert (\kund) \cap [ h,+\infty ) \right) 
= (k_i - i') + (k_{i+1} - 1) + \cdots +(k_{a} - 1).
\e
If $\kund = \kund^1 \sharp \kund^0$, recall that $\Vert (\kund) \simeq \Vert (\kund^0) \cup \Vert (\kund^1)$. Since these are ordered sets, it corresponds to a permutation of $\left\lbrace 1, \ldots , v(\kund) \right\rbrace$. Observe that its signature is given by $(-1)^{\heartsuit^{\kund^1}_{\kund^0}}$, with
\e\label{eq:heartsuit}
\heartsuit^{\kund^1}_{\kund^0} = \sum_{h = 1}^{a^1} (k^1_{h} -1) v_{\geq h}(\kund^0).
\e
Indeed,  one can go from  $\Vert (\kund^0) \cup \Vert (\kund^1) $ to $\Vert (\kund)$ by moving successively the vertices of $\Vert (\kund^1)$ corresponding to the tree $h$ in front of those of $\kund^0$ at the right of $h$: this contributes to $(k^1_{h} -1) v_{\geq h}(\kund^0)$. It follows that:

\begin{lemma}\label{lem:signs_glueing_forests}Assume $\kund = \kund^1 \sharp \kund^0$ and $\lund = \lund^0 \sharp \lund^1$, then, as oriented vector spaces:
\ea
U^{\kund} &\simeq (-1)^{\heartsuit^{\kund^1}_{\kund^0}} \cdot U^{\kund^0}  \times U^{\kund^1}, \label{eq:sign_glueing_U}\\
D_{\lund} &\simeq  (-1)^{\heartsuit^{\lund^0}_{\lund^1} + v(\lund^0) v(\lund^1)}  \cdot D_{\lund^0}  \times D_{\lund^1},\label{eq:sign_glueing_D}\\
J^{\kund}_{\lund} &\simeq  (-1)^{\heartsuit^{\kund^1}_{\kund^0} + \heartsuit^{\lund^0}_{\lund^1} + v(\lund^0) ( v(\kund^1) + v(\lund^1) )}  \cdot  J^{\kund^0}_{\lund^0}\times J^{\kund^1}_{\lund^1}\label{eq:sign_glueing_J}
\ea

\end{lemma}
\begin{proof}Equation (\ref{eq:sign_glueing_U}) follows from  (\ref{eq:heartsuit}). 
Identification (\ref{eq:sign_glueing_D}) follows from (\ref{eq:sign_glueing_U}), plus exchanging positions of $\lund^0$ and $\lund^1$, which adds $v(\lund^0) v(\lund^1)$. 
Finally, from (\ref{eq:sign_glueing_D}) and (\ref{eq:sign_glueing_U}):
\ea
J^{\kund}_{\lund} &\simeq U^{\kund} \times D_{\lund}\\
 &\simeq (-1)^{\heartsuit^{\kund^1}_{\kund^0} + \heartsuit^{\lund^0}_{\lund^1} + v(\lund^0) v(\lund^1)} \cdot U^{\kund^0}  \times U^{\kund^1} \times D_{\lund^0}  \times D_{\lund^1} \nonumber
\ea
which proves (\ref{eq:sign_glueing_J}) after exchanging $U^{\kund^1}$ and $D_{\lund^0}$.
\end{proof}

\begin{prop}\label{prop:or_glueing_maps} With $x^i = v(\kund^i)$, $y^i = v(\lund^i)$, $\heartsuit_k = \heartsuit^{\kund^1}_{\kund^0}$, $\heartsuit_l = \heartsuit^{\lund^0}_{\lund^1}$, the orientations of the gluing maps are given by:
\ea
{\rho} &= \heartsuit_k + \heartsuit_l + y^0 (x^1 + y^1) + x^0 + y^0 +1 ,\label{eq:sign_glueing_BF} \\
{\rho_0} &= \heartsuit_k + \heartsuit_l + y^0 (x^1 + y^1) + x^0 + y^0 \label{eq:sign_glueing_GBF_0} \\
{\rho_1} &= \heartsuit_k + \heartsuit_l + y^0 (x^1 + y^1) +1 .\label{eq:sign_glueing_GBF_1}
\ea

\end{prop}
\begin{proof} 
Let $K = {K}^{\kund}_{\lund}$, $K^i = {K}^{\kund^i}_{\lund^i}$, $i=0, 1$; and likewise let $J = {J}^{\kund}_{\lund}$, $J^i = {J}^{\kund^i}_{\lund^i}$.

Start with $\rho_1$, and assume that $c(\kund^0, \lund^0)=1$. The gluing map is given by composing the isomorphisms
\ea
J &\simeq J^0 \oplus J^1 \\
 &\simeq \rr (-v^0)  \oplus K^0 \oplus J^1 .
\ea
From (\ref{eq:sign_glueing_J}), the isomorphism on the first line affects sign by $\heartsuit_k + \heartsuit_l + y^0 (k^1 + y^1)$. The second line changes the sign, therefore we get $\rho_1$.

For $\rho_0$, assuming $c(\kund^1, \lund^1)=1$, The corresponding gluing map is given by:
\ea
J &\simeq J^0 \oplus J^1 \\
 &\simeq   J^0  \oplus \rr v^1 \oplus K^1 \\
 &\simeq   \rr v^1 \oplus J^0 \oplus K^1 ,
\ea
which affects signs by $\heartsuit_k + \heartsuit_l + y^0 (x^1 + y^1) + \dim J^0 $.

For $\rho$, write 
\ea
\rr (v^0 + v^1) \oplus K &\simeq    \rr v^0 \oplus K^0  \oplus \rr v^1 \oplus K^1 \\
 &\simeq    \rr v^0 \oplus \rr v^1 \oplus K^0   \oplus K^1 \\
 &\simeq    \rr (v^0 + v^1)\oplus \rr v^1 \oplus K^0   \oplus K^1 ,
\ea
which gives  $\rho = \heartsuit_k + \heartsuit_l + y^0 (x^1 + y^1) + \dim K^0 + 0$.
\end{proof}

\subsection{Higher codimension corner strata}
\label{ssec:higher_codim_corner_strata}

The partial compactifications $\left( \overline{K}^{\kund}_{\lund} \right)_{\leq 1}$ and $\left( \overline{J}^{\kund}_{\lund} \right)_{\leq 1}$ are noncompact manifolds with boundary. One can compactify them by adding inductively higher corner strata: typically, if $\kund = \kund^n \sharp \cdots \sharp \kund^0$ and $\lund = \lund^0 \sharp \cdots \sharp \lund^n$, one should have  \emph{in good cases} some gluing maps of the form
\e
\left( (-\infty, + \infty]^c \setminus \left\lbrace + \infty \right\rbrace^c \right) \times {K}^{\kund^0}_{\lund^0} \times \cdots  \times {K}^{\kund^n}_{\lund^n}  \to \left( \overline{K}^{\kund}_{\lund} \right)_{\leq c- 1},
\e
where in most cases $c=n$, and then form inductively
\e
 \left( \overline{K}^{\kund}_{\lund} \right)_{\leq c} = \left( \left( \overline{K}^{\kund}_{\lund} \right)_{\leq c- 1} \cup \bigcup \left( (-\infty, + \infty]^c \setminus \left\lbrace + \infty \right\rbrace^c \right) \times {K}^{\kund^0}_{\lund^0} \times \cdots  \times {K}^{\kund^n}_{\lund^n}  \right) /\sim .
\e
However, this is slightly more subtle than one might think. In particular, for some choices of multi-indices one has $c<n$, and the local corner structure might be more complicated than the standard corner $(-\infty, + \infty]^c$, as is already the case for the multiplihedra $J_4 \simeq \overline{K}^4_2$, see for ex.  \cite[Fig.~11]{MauWoodwardrealiz}. More details about these subtleties can be found in the first arXiv version of this paper.

\begin{remark}\label{rem:generalized_bimult} Both $K=J(0)$ and $J= J(1)$ are part of a larger family $J(n)$ of \emph{$n$-grafted $f$-multiplihedra}, defined as moduli spaces of biforests with $n$ grafting levels. These  would correspond to sequences of group morphisms
\e
\xymatrix{G_0 \ar[r]^{\phi_1} & G_1 \ar[r]^{\phi_2} & G_2  \ar[r]^{\phi_3}   & \cdots  \ar[r]^{\phi_n}& G_n}.
\e
Its boundary strata would include products of the form
\e
J(n_0)^{\kund^0}_{\lund^0} \times \cdots \times J(n_m)^{\kund^m}_{\lund^m} ,
\e
with
\ea
\kund &= \kund^m \sharp \cdots \sharp \kund^0, \\
\lund &= \lund^0 \sharp \cdots \sharp \lund^m, \\
n &= n_0 +\cdots + n_m .
\ea

In particular, the codimension one strata is given by:
\e
\partial^1 J(n)^{\kund}_{\lund} =  (n-1) J(n-1)^{\kund}_{\lund}\sqcup \coprod_{\begin{subarray}{c}\kund = \kund^1 \sharp \kund^0  \\ \lund = \lund^0 \sharp \lund^1 \\ n = n_0+n_1 \end{subarray}} J(n_0)^{\kund^0}_{\lund^0} \times  J(n_1)^{\kund^1}_{\lund^1}.
\e
The first part corresponds to grafting levels merging together, and the second part corresponds to breaking. 
The first author will introduce these spaces in more detail in \cite{morse_fun} in order to construct the higher categorical structure of $\fBialg$, and $\infty$-functors from the Morse complex.

\end{remark}

\section{Forest bialgebras, bimodules and morphisms}
\label{sec:bialg_bimod_mph}
We now detail the algebraic structures induced by $\overline{K}_{\leq 1}$ and  $\overline{J}_{\leq 1}$. Let us first set some notations and conventions.
\subsection{Notations and conventions}
\label{ssec:notat_convent}

Let $A$ be a chain complex, that we assume absolutely graded over $\zz$, or possibly relatively graded over some $\Z{2N}$, as long as it is absolutely graded over $\Z{2}$ (except if the ground ring has characteristic 2).  
For concision,  $A^n$ will always stand for $A^{\otimes n} = A\otimes \cdots \otimes A$, equipped with the tensor product differential
\e
\partial_{A^n} = \sum_{i=1}^n  id^{i-1}\otimes \partial \otimes id^{n-i}
\e
that we still denote $\partial$.

The maps we will consider will have several inputs (on top) and outputs (at the bottom), which we think as lying in rectangular grids as in Figure~\ref{fig:rectangular_box}. We will take the convention to order the edges by lexicographical order: at a given level, let the edges of the ascending and descending forests be respectively $e_1, \ldots , e_a$ and $f_1, \ldots , f_b$. Then the edges of the fiber product graph are ordered as:
\e
e_1 f_1,  e_1 f_2, \ldots , e_1 f_b , e_2 f_1, \ldots ,e_2 f_b, \ldots ,e_a f_1,   \ldots , e_a f_b  .
\e

\begin{figure}[!h]
    \centering
    \def\svgwidth{.50\textwidth}
    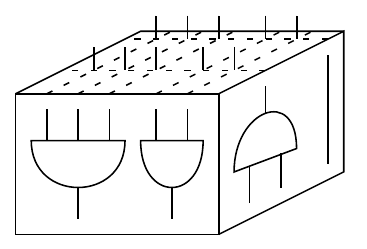
      \caption{The maps $\alpha^{\kund}_{\lund}$ as rectangular boxes.}
      \label{fig:rectangular_box}
\end{figure}

This means that we consider the ``rectangular tensor product'' chain complex $(A^b)^a$. It is identified with $A^{ba}$. When splitting $a=a_1+ \cdots +a_k$ we will use the obvious identification
\e\label{eq:iso_split_a}
(A^b)^a \simeq (A^b)^{a_1} \otimes \cdots \otimes (A^b)^{a_k}.
\e
However, when splitting $b=b_1+ \cdots +b_k$, the identification 
\e\label{eq:iso_split_b}
(A^b)^a \simeq (A^{b_1})^{a} \otimes \cdots \otimes (A^{b_k})^{a}
\e
involves some mixing. Let $\tau^a_b\colon (A^b)^a \to (A^a)^b$ stand for the map induced by the obvious permutation, with sign given by the Koszul sign convention. Without the sign, this corresponds to the permutation $\sigma(a,b)$ appearing in Markl's fractions \cite{Markl_resolution}. 
More explicitly:

\begin{prop}The map  $\tau^a_b\colon (A^b)^a \to (A^a)^b$ is given, on homogeneous elements, by:
\ea
\tau^a_b & \left( (x_{11} \otimes \cdots \otimes x_{1b} ) \otimes \cdots \otimes (x_{a1} \otimes \cdots \otimes x_{ab} ) \right) \label{eq:tau_a_b_init} \\
= & (-1)^{\clubsuit^a_b } \left( (x_{11} \otimes \cdots \otimes x_{a1} ) \otimes \cdots \otimes (x_{1b} \otimes \cdots \otimes x_{ab} ) \right),\nonumber
\ea
with
\e
\clubsuit^a_b  = \sum_{\begin{subarray}{c} 1 \leq i'<i \leq a \\ 1\leq j<j' \leq b \end{subarray}} \abs{x_{ij}} \abs{x_{i'j'}}. 
\e
\end{prop}
\begin{proof} Start with the initial configuration in (\ref{eq:tau_a_b_init}), and move the vectors one by one to the final one in the order that they appear in that final configuration. First, $x_{ij}=x_{11}$ is already at the right place : it is exchanged with sign contribution
\e
0 =  \sum_{\begin{subarray}{c} i'<1  \\ 1<j'  \end{subarray}} \abs{x_{11}} \abs{x_{i'j'}}.
\e
Next, $x_{ij}=x_{21}$ is exchanged with $x_{12} \otimes \cdots \otimes x_{1b}$, which affects the sign by
\e
 \sum_{\begin{subarray}{c} i'<2  \\ 1<j'  \end{subarray}} \abs{x_{12}} \abs{x_{i'j'}}.
\e
Then  $x_{ij}=x_{31}$ is exchanged with $(x_{12} \otimes \cdots \otimes x_{1b}) \otimes (x_{22} \otimes \cdots \otimes x_{2b}) $, since $x_{21}$ has already been moved. This adds
\e
 \sum_{\begin{subarray}{c} i'<3  \\ 1<j'  \end{subarray}} \abs{x_{13}} \abs{x_{i'j'}}.
\e
In general, $x_{ij}$ moves in front of the rectangle of $x_{i'j'}$, with  $ i'<i$, $j<j'$.
\end{proof}

Then the identification (\ref{eq:iso_split_b}) is as in the diagram:
\e
\xymatrix{ (A^b)^a \ar[d]^{\tau^a_b} \ar[r] & (A^{b_1})^{a} \otimes \cdots \otimes (A^{b_k})^{a} \\
(A^a)^b \ar[r]  & (A^{a})^{b_1} \otimes \cdots \otimes (A^{a})^{b_k} \ar[u]_{\tau^{b_1}_a \otimes \cdots \otimes \tau^{b_k}_a }}
\e

We will also need to move apart some blocks. If $a=a_1+a_2+a_3$, we will identify 
\e
(A^b)^a\simeq (A^b)^{a_1 + a_3} \otimes (A^b)^{a_2} 
\e
via $\xund_1 \otimes \xund_2 \otimes \xund_3 \mapsto (-1)^{\abs{\xund_2}\abs{\xund_3}} (\xund_1 \otimes \xund_3 ) \otimes \xund_2$ Likewise, if $b=b_1+b_2+b_3$, we will identify 
\e
(A^b)^a\simeq (A^{b_1 + b_3})^a \otimes (A^{a_2})^a 
\e 
similarly, after conjugating with $\tau^a_b$ maps.

Given maps $P_i\colon (A^b)^{a_i} \to (A^d)^{c_i}$ and  $P\colon (A^b)^{a} \to (A^d)^{c}$, with $a = a_1+ \cdots + a_k$ and $c = c_1+ \cdots + c_k$, whenever we write $P \simeq P_1 \otimes \cdots \otimes P_k$, we will mean that the following diagram commutes, where the vertical maps are the identifications (\ref{eq:iso_split_a}):
\e
\xymatrixcolsep{.8in}
\xymatrix{ (A^b)^{a} \ar[d]\ar[r]^P & (A^d)^{c} \ar[d]\\
(A^b)^{a_1} \otimes \cdots \otimes (A^b)^{a_k} \ar[r]^{P_1 \otimes \cdots \otimes P_k} & (A^d)^{c_1} \otimes \cdots \otimes (A^d)^{c_k} 
}
\e
Likewise for maps $P_i\colon (A^{b_i})^{a} \to (A^{d_i})^{c}$ and  $P\colon (A^b)^{a} \to (A^d)^{c}$, with $b = b_1+ \cdots + b_k$, $d = d_1+ \cdots + d_k$, and the identifications (\ref{eq:iso_split_b}).
\subsection{$f$-bialgebras}
\label{ssec:f_bialg}

We can now define the forest version of a bialgebra:

\begin{defi}\label{def:f_bialg} An  \emph{$f$-bialgebra} is a chain complex $A$ with a collection of operations 
\e
{\alpha}^{\kund}_{\lund}\colon (A^b)^{\abs{\kund}} \to (A^{\abs{\lund}})^{a}
\e
of degree $\deg {\alpha}^{\kund}_{\lund} = \dim K^{\kund}_{\lund} = v(\kund) + v(\lund)- 1$, such that: 
\begin{itemize}
\item They satisfy the family of coherence relations:
\e\label{eq:R_kund_lund}\tag{$R^{\kund}_{\lund}$}
0= \sum_{\begin{subarray}{c} \kund = \kund^1 \sharp \kund^0 \\ \lund = \lund^0 \sharp \lund^1 \end{subarray}} (-1)^{\rho} \cdot  {\alpha}^{\kund^0}_{\lund^0}\circ {\alpha}^{\kund^1}_{\lund^1} ,
\e
where $\rho$ is given in Proposition~\ref{prop:or_glueing_maps}. 
It follows from $(R^{\One_a}_{\One_b})$ that $\alpha^{\One_a}_{\One_b}$ is a differential on $(A^b)^a$. 

\item For $a,b\geq 1$, $\alpha^{\One_a}_{\One_b}$ coincides with the tensor product differential induced by the differential $\partial_A$ of $A$. We will refer to this condition as $(T^a_b)$.
\item If $\kund = \One_a$ and $\lund = (1, \ldots 1, l_j , 1, \ldots, 1)$, with $l_j\geq 2$, then
\e\label{eq:V_a_lund}\tag{$V^{a}_{\lund}$}
{\alpha}^{\kund}_{\lund} \simeq id_{(A^{j-1})^a} \otimes {\alpha}^{\kund}_{l_j} \otimes id_{(A^{b-j})^a},
\e
where $\simeq$ is in the sense of (\ref{eq:iso_split_b}).
\item Likewise, if  $\lund = \One_b$ and $\kund = (1, \ldots 1, k_i , 1, \ldots, 1)$, with $k_i\geq 2$, then
\e\label{eq:V_kund_b}\tag{$V^{\kund}_{b}$}
{\alpha}^{\kund}_{\lund} \simeq id_{(A^{b})^{i-1}} \otimes {\alpha}^{k_i}_{\lund} \otimes id_{(A^{b})^{a-i}},
\e
where $\simeq$ is in the sense of (\ref{eq:iso_split_a}).
\item If either $\kund = \One_a$ and $\lund$ has at least two entries $\geq 2$, or vice versa, then ${\alpha}^{\kund}_{\lund} =0$. We will refer to these conditions as $(V^{a}_{\lund})$ and $(V^{\kund}_{b})$ as well.

\item (vertical tree deletion) Suppose that $\lund$ is an almost vertical forest, i.e. only with 1 and 2. Let $\kund$ have a vertical tree at position $i$, i.e. $k_i = 1$, and let $\widehat{\kund} = (k_1, \ldots , \widehat{k_i} , \ldots, k_a)$, then
\e\label{eq:D_kund_i_lund}\tag{$D^{\kund,i}_{\lund}$}
{\alpha}^{\kund}_{\lund} \simeq {\alpha}^{\widehat{\kund}}_{\lund} \otimes  {\widetilde{\alpha}}^{1}_{\lund},
\e
where ${\widetilde{\alpha}}^{1}_{\lund} =  {\widetilde{\alpha}}^{1}_{l_1} \otimes \cdots \otimes {\widetilde{\alpha}}^{1}_{l_b}$, with ${\widetilde{\alpha}}^{1}_{1}= id_A$ and ${\widetilde{\alpha}}^{1}_{2}= \alpha^1_2$.

Likewise, if $\kund$ is an almost vertical forest. Let $\lund$ have a vertical tree at position $i$, i.e. $l_i = 1$, and let $\widehat{\lund} = (l_1, \ldots , \widehat{l_i} , \ldots, l_b)$, then
\e\label{eq:D_kund_lund_i}\tag{$D^{\kund}_{\lund,i}$}
{\alpha}^{\kund}_{\lund} \simeq {\alpha}^{\kund}_{\widehat{\lund}} \otimes  \widetilde{\alpha}^{\kund}_{1},
\e
where ${\widetilde{\alpha}}^{\kund}_{1} =  {\widetilde{\alpha}}^{k_1}_{1} \otimes \cdots \otimes {\widetilde{\alpha}}^{k_a}_{}$, with ${\widetilde{\alpha}}^{1}_{1}= id_A$ and ${\widetilde{\alpha}}^{2}_{1}= \alpha^2_1$.

\end{itemize}
\end{defi}

\begin{remark}The maps ${\widetilde{\alpha}}^{\kund}_{\lund}$ come from the spaces $\widetilde{K}^{\kund}_{\lund}$.
\end{remark}

\begin{remark}The formula for $\rho$ in Proposition~\ref{prop:or_glueing_maps}  give the orientation of the gluing map under the assumption that $c(\kund, \lund) = c(\kund^0, \lund^0) = c(\kund^1, \lund^1) = 1$. Nevertheless, the expression makes sense for any $\kund^0, \lund^0,\kund^1, \lund^1$. In particular, when $c(\kund^0, \lund^0) = 0$ (resp. $c(\kund^1, \lund^1) = 0$), $\rho=1$ (resp. $\rho=v(\kund) + v(\lund) +1 =  \deg \alpha^{\kund}_{\lund} \text{ mod }2$). 
Since we have assumed $\alpha^{\kund}_{\lund} = 0$ for $c(\kund, \lund) \geq 2$, (\ref{eq:R_kund_lund}) can be written equivalently as:
\e
\partial( \alpha^{\kund}_{\lund} ) =  \sum_{\begin{subarray}{c} c(\kund^0, \lund^0)  = 1\\ c(\kund^1, \lund^1) = 1 \end{subarray}} (-1)^{\rho} \cdot  {\alpha}^{\kund^0}_{\lund^0}\circ {\alpha}^{\kund^1}_{\lund^1} ,
\e
Where $\partial$ is the differential on $\mathrm{Hom}( (A^b)^{\abs{\kund}} , (A^{\abs{\lund}})^{a})$ induced by $\partial_A$, and the right hand side  corresponds to summing over the faces of $\partial \left( \overline{K}^{\kund}_{\lund}\right)_{\leq 1}$.

\end{remark}

\begin{prop}
\label{prop:rel_f_bialg_A_inf} Let $(A, {\alpha}^{\kund}_{\lund})$ be an $f$-bialgebra. 
Then, with $m^k = {\alpha}^{k}_{1}$, $(A,m^k)$ is an \Ainf -algebra (with the sign convention below), i.e.
\e\label{eq:Ainf_alg_rel}
0 = \sum_{\begin{subarray}{c} k+1 = k^0 + k^1 \\ 1\leq i \leq k^0 \end{subarray}} (-1)^{k^0 k^1 + i(k^1 +1)} \cdot m^{k^0} \circ \left(id^{\otimes (i-1)} \otimes m^{k^1} \otimes id^{\otimes (k^0 -i)}   \right) .
\e
Likewise, if $A$ is an $f$-bialgebra, 
then with $\delta^l = \alpha^1_l$, $(A,\delta^l)$ is an \Ainf -coalgebra (with the sign convention below), i.e.
\e\label{eq:Ainf_coalg_rel}
0 = \sum_{\begin{subarray}{c} l+1 = l^0 + l^1 \\ 1\leq j \leq l^1 \end{subarray}} (-1)^{(l^0 +1)j +1} \cdot \left(id^{\otimes (j-1)} \otimes \delta^{l^0} \otimes id^{\otimes (l^1 -j)}   \right) \circ \delta^{l^1} .
\e
\end{prop}
\begin{remark} More generally, one can show that $A^b$, endowed with $m^k = \alpha^k_{\One_b}$, is an \Ainf -algebra in the above sense. Likewise,  $A^a$, endowed with $\delta_l = \tau^a_l \circ \alpha^{\One_a}_{l}$, is an \Ainf -coalgebra.
\end{remark}
\begin{proof}
Take $\kund = (k)$ and $\lund = (1)$. If $\kund =\kund^1  \sharp \kund^0 $ and $ \lund = \lund^0  \sharp \lund^1 $, then one must have $\lund^0 =\lund^1= (1)$, $\kund^0 = (k^0)$, and condition $c(\kund^1, \lund^1)=1$ implies that $\kund^1$ is of the form $\kund^1 = (1, \ldots , 1, k^1, 1, \ldots , 1)$, with $k^1$ at position $i$, $1\leq i \leq k^0 $. One then has:
\ea
v(\lund^0) &= v(\lund^1) = 0, \\
v(\kund^0) &= k^0 -1, \
v(\kund^1) = k^1 -1 , \\
\heartsuit^{\lund^0}_{\lund^1} &= 0,\
\heartsuit^{\kund^1}_{\kund^0} = (k^1 -1)(k^0-i).
\ea
From Proposition~\ref{prop:or_glueing_maps}, it follows that $\rho= k^0 k^1 + i(k^1 +1)\text{ mod }2$.
Finally, $(V^{\kund^1}_1)$ implies that 
\e
\alpha^{\kund^1}_1 =  \left( id^{\otimes (i-1)} \otimes m^{k^1} \otimes id^{\otimes (k^0 -i)}  \right).
\e
The \Ainf -relation (\ref{eq:Ainf_alg_rel}) follows. The proof of the \Ainf -coalgebra statement is analogous.
\end{proof}

Let us now describe $(R^2_2)$ more explicitly. Assume now $\kund = \lund = (2)$.  If $(\kund, \lund) =(\kund^1, \lund^1)  \sharp (\kund^0, \lund^0) $, then (excluding vertical floors) one must have either
\ea
\kund^1 = \lund^0 = (2) ,\\
\kund^0 = \lund^1 = (1) ,
\ea
or
\ea
\kund^1 = \lund^0 = (1,1) ,\\
\kund^0 = \lund^1 = (2) .
\ea
In the first case,  $\rho= 1$; and in the second,  $\rho= 0$. Furthermore, $(D^{(1,1);1}_{(2)})$ and $(D^{(2)}_{(1,1);1})$ give respectively:
\ea
\alpha^{(1,1)}_{(2)} &=\alpha^{1}_{2}  \otimes \alpha^{1}_{2} ,\\
\alpha^{(2)}_{(1,1)} &= (\alpha^{2}_{1} \otimes \alpha^{2}_{1})\circ \tau^2_2 . 
\ea
Therefore, $(R^2_2)$ reads:
\e
\partial (\alpha^2_2) = -\alpha^{1}_{2}  \circ \alpha^{2}_{1} +  (\alpha^{2}_{1} \otimes \alpha^{2}_{1})\circ \tau^2_2 \circ (\alpha^{1}_{2}  \otimes \alpha^{1}_{2}),
\e
which is indeed the Hopf relation of bialgebras up to homotopy, as expected. It follows that the homology of an $f$-bialgebra is a bialgebra. Furthermore:

\begin{prop}
\label{prop:rel_dg_bialg} Say that $(A, {\alpha}^{\kund}_{\lund})$ is a \emph{dg $f$-bialgebra} if $ {\alpha}^{\kund}_{\lund}= 0$ unless $\deg  {\alpha}^{\kund}_{\lund} \leq 0$. 
Then, dg $f$-bialgebras are in one-to-one correspondence with dg bialgebras.
\end{prop}

\begin{proof}If $(A, {\alpha}^{\kund}_{\lund})$ is a \emph{dg $f$-bialgebra}, then it follows from $(R^2_1)$, $(R^3_1)$, $(R^1_2)$, $(R^1_3)$ and  $(R^2_2)$ that $(A, \partial, \alpha^2_1,\alpha^1_2)$ is a dg bialgebra. Conversely, given a dg bialgebra, relations $(V)$ and $(D)$ prescribe the degree $\leq 0$ maps. Setting the positive degree  operations equal to zero gives a dg $f$-bialgebra.

\end{proof}

In many cases, some maps $\alpha^{\kund}_{\lund}$ must vanish for degree reasons:
\begin{prop}
\label{prop:vanishing_alpha_deg}  If $A$ is $\zz$-graded and is concentrated in degrees 0, ..., $\mathrm{maxdeg}A$, then $\alpha^{\kund}_{\lund}=0$  if $\abs{\kund} > (a \cdot \mathrm{maxdeg} A -1)\abs{\lund} +a+b+1.$
\end{prop}

\begin{proof}${\alpha}^{\kund}_{\lund}\colon (A^b)^{\abs{\kund}} \to (A^{\abs{\lund}})^{a}$ has degree $\deg {\alpha}^{\kund}_{\lund}  = v(\kund) + v(\lund)- 1 = \abs{\kund} -a + \abs{\lund} -b -1 $, therefore it is zero if
\e
\deg {\alpha}^{\kund}_{\lund} + \mathrm{mindeg}(A^b)^{\abs{\kund}}  > \mathrm{maxdeg}(A^{\abs{\lund}})^{a}.
\e
Furthermore, $\mathrm{mindeg}(A^b)^{\abs{\kund}} \geq 0$ and $\mathrm{maxdeg}(A^{\abs{\lund}})^{a} = a \abs{\lund} \mathrm{maxdeg}A$. The claim follows.
\end{proof}

\subsection{Morphisms of $f$-bialgebras}
\label{ssec:mph_Ainf_bialg}

Just like the coherence relations of \Ainf -morphisms come from the boundary structure of the multiplihedron, our notion of morphisms between $f$-bialgebras comes from the boundary structure of $J^{\kund}_{\lund}$. 

\begin{defi}\label{def:mph_f_bialg} Let $(A, \alpha^{\kund}_{\lund})$ and  $(B, {\beta}^{\kund}_{\lund})$ be two $f$-bialgebras. A morphism $f\colon A \to B$ is a collection of operations:
\e
f^{\kund}_{\lund} \colon  (A^b)^{\abs{\kund}} \to (B^{\abs{\lund}})^{a}
\e
satisfying the family of coherence relations: 
\e\label{eq:M_kund_lund_mph}\tag{$M^{\kund}_{\lund}$}
0 = \sum_{  \begin{subarray}{c} \kund = \kund^1 \sharp \kund^0 \\ \lund = \lund^0 \sharp \lund^1 \end{subarray}} (-1)^{\rho_0} \cdot  f^{\kund^0}_{\lund^0}\circ {\alpha}^{\kund^1}_{\lund^1}  +  (-1)^{\rho_1} \cdot  {\beta}^{\kund^0}_{\lund^0}\circ f^{\kund^1}_{\lund^1},
\e
Furthermore, $f$ satisfies the simplification relations:
\begin{itemize}
\item If $\lund$ is vertical, then
\e\label{eq:W_kund_b}\tag{$W^{\kund}_{b}$}
f^{\kund}_{\lund}  \simeq 
f^{k_1}_{\lund}  \otimes \cdots \otimes  f^{k_a}_{\lund}.
\e

\item If $\kund$ is vertical, then
\e\label{eq:W_a_lund}\tag{$W^{a}_{\lund}$}
f^{\kund}_{\lund}  \simeq  f^{\kund}_{l_1} \otimes \cdots \otimes f^{\kund}_{l_b} .
\e
\end{itemize}

\end{defi}

\begin{remark}
Since ${\alpha}^{\kund^1}_{\lund^1}  = 0$ if $c(\kund^1, \lund^1 )\geq 2$ and  ${\beta}^{\kund^0}_{\lund^0} = 0$ if $c(\kund^0 , \lund^0)\geq 2$, one can write (\ref{eq:M_kund_lund_mph}) as:
\e
\partial (f^{\kund}_{\lund}) = \sum_{ c(\kund^1, \lund^1 ) =1 } (-1)^{\rho_0} \cdot  f^{\kund^0}_{\lund^0}\circ {\alpha}^{\kund^1}_{\lund^1}  +  \sum_{  c(\kund^0, \lund^0 )=1} (-1)^{\rho_1} \cdot  {\beta}^{\kund^0}_{\lund^0}\circ f^{\kund^1}_{\lund^1}.
\e
\end{remark}

These induce \Ainf -morphisms:
\begin{prop}
\label{prop:Ainf_mph_from_f_mph} Consider $f\colon A \to B$ as in Definition~\ref{def:mph_f_bialg} above. Then the family $\varphi^k = f^k_1$ defines an \Ainf -morphism from $A$ to $B$ (in the sign convention below), endowed with the \Ainf -algebra structures $m_A^k$, $m_B^k$ of Proposition~\ref{prop:rel_f_bialg_A_inf}, namely they satisfy:
\ea\label{eq:Ainf_alg_mph_rel}
0 &= \sum_{\begin{subarray}{c} k+1 = k^0+ k^1,\\ 1\leq i \leq k^0 \end{subarray}} (-1)^{k^1 k^0 + (k^1+1)i +1} \cdot \varphi^{k^0} \circ \left( id^{\otimes (i-1)} \otimes m_A^{k^1} \otimes id^{\otimes (k^0 -i)}   \right) \\
&+ \sum_{\begin{subarray}{c} \abs{\kund^1} = k,\\ n(\kund^1)= k^0 \end{subarray}} (-1)^{1+\sum_{i=1}^{k^0} (k^1_i +1 )(k^0 +i) } \cdot  m_B^{k^0} \circ \left( \varphi^{k^1_1} \otimes \cdots \otimes \varphi^{k^1_{k^0}}   \right). \nonumber
\ea
Likewise, $\psi^l = f^1_l$ defines an \Ainf -morphism from $A$ to $B$, endowed with the \Ainf -coalgebra structures $\delta_A^l$, $\delta_B^l$ of Proposition~\ref{prop:rel_f_bialg_A_inf}, namely they satisfy:
\ea\label{eq:Ainf_coalg_mph_rel}
0 &=  \sum_{\begin{subarray}{c} l+1 = l^0+ l^1,\\ 1\leq j \leq l^1 \end{subarray}} (-1)^{l^0  j+ l^0 +j  } \cdot \left( id^{\otimes (j-1)} \otimes \delta_B^{l^0} \otimes id^{\otimes (l^1 - j)}   \right) \circ \psi^{l^1} \\
&+\sum_{\begin{subarray}{c} \abs{\lund^0} = l,\\ n(\lund^0)= l^1 \end{subarray}} (-1)^{\sum_{j=1}^{l^1} j(l^0_j +1) } \cdot  \left( \psi^{l^0_1} \otimes \cdots \otimes \psi^{l^0_{l^1}}  \right) \circ  \delta_A^{l^1} .\nonumber
\ea
\end{prop}

\begin{proof} From the axioms, ${\alpha}^{\kund^1}_{\lund^1}  = 0$ unless $c(\kund^1, \lund^1 )\leq 1$. Here, $\lund = (1)$, which implies that $\kund^1$ is of the form $\kund^1 = (1, \ldots , 1, k^1, 1, \ldots , 1)$, with $k^1$ at position $i$, $1\leq i \leq k^0 $. One then has:
\ea
v(\lund^0) &= v(\lund^1) = 0, \\
v(\kund^0) &= k^0 -1, \
v(\kund^1) = k^1 -1 , \\
\heartsuit^{\lund^0}_{\lund^1} &= 0,\
\heartsuit^{\kund^1}_{\kund^0} = (k^1 -1)(k^0-i).
\ea
From Proposition~\ref{prop:or_glueing_maps}, it follows that $\rho_0=  k^0 k^1 + i(k^1 +1) +1 \text{ mod }2$.
Finally, $(V^{\kund^1}_1)$ gives 
\e
\alpha^{\kund^1}_1 =  \left( id^{\otimes (i-1)} \otimes m_A^{k^1} \otimes id^{\otimes (k^0 -i)}  \right).
\e

Consider now the second sum. As for the first one, $\lund^0 = \lund^1 = (1)$, $\kund^0= (k^0)$, but now there are no constraints on $\kund^1$, other than $\abs{\kund^1} = k$ and $n(\kund^1)= k^0 $. One has:
\ea
v(\lund^0) &= v(\lund^1) = 0, \\
v(\kund^0) &= k^0 -1, \
v(\kund^1) = \sum_{i=1}^{k^0} (k^1_i -1) , \\
\heartsuit^{\lund^0}_{\lund^1} &= 0,\
\heartsuit^{\kund^1}_{\kund^0} = \sum_{i=1}^{k^0} (k^1_i -1)(k^0-i),
\ea
and therefore, from Proposition~\ref{prop:or_glueing_maps}:
\e
\rho_1 = 1+\sum_{i=1}^{k^0} (k^1_i +1 )(k^0 +i)  .
\e
Using  $(W^{\kund^1}_{1})$, we get
\e
f^{\kund^1}_{1} = \varphi^{k^1_1} \otimes \cdots \otimes \varphi^{k^1_{k^0}},
\e
which proves (\ref{eq:Ainf_alg_mph_rel}). The proof of (\ref{eq:Ainf_coalg_mph_rel}) is similar.
\end{proof}

\begin{remark}(Composition) 
Morphisms of $f$-bialgebras can be composed: if $f\colon A \to B$ an $g\colon B\to C$, then $g\circ f$ is defined by
\e
(g\circ f)^{\kund}_{\lund} = \sum_{\begin{subarray}{c}\kund = \kund^1 \sharp \kund^0  \\ \lund = \lund^0 \sharp \lund^1 \end{subarray}} \pm g^{\kund^0}_{\lund^0} \circ f^{\kund^1}_{\lund^1},
\e
with the sign given by the orientation of the boundary strata of $J(2)^{\kund}_{\lund}$. With this operation, $f$-bialgebras form a category. However, mapping a Lie group to its Morse complex does not define a strict functor: $(g\circ f)_* = g_* \circ f_*$ only holds  up to a homotopy, measured by $J(2)$. 
In \cite{morse_fun}, the first author will show that $\fBialg$ can be enhanced to an $(\infty,1)$-category, using the higher multiplihedron $J(n)$. With this structure, the above assignment indeed defines an $\infty$-functor.
\end{remark}

\subsection{Bimodules of $f$-bialgebras}
\label{ssec:bimod_Ainf_bialg}
Recall that for two \Ainf -algebras $(A, B)$,  an \Ainf -bimodule $M$ consists in a family of operations 
\e
\mu^{k^l|1|k^r}\colon (A)^{k^l}\otimes M\otimes (B)^{k^r} \to M,
\e
satisfying the \Ainf -relations. For $f$-bialgebras, we will have a family indexed by multi-indices $\mu^{\kund^l|1|\kund^r}_{\lund}$ satisfying relations analogous to (\ref{eq:R_kund_lund}), plus extra relations analogous to (\ref{eq:V_a_lund}), (\ref{eq:V_kund_b}) and (\ref{eq:D_kund_i_lund}),  (\ref{eq:D_kund_lund_i}). In order to express the vertical tree deletion (\ref{eq:D_kund_i_lund}) when $i$ corresponds to the $M$ input, one needs to introduce extra maps with no $M$ inputs, that can not always be expressed from the structure operations of the given pair of bialgebras. We will denote these operations $\mu^{\kund^l|0|\kund^r}_{\lund}$.

\begin{figure}[!h]
    \centering
    \def\svgwidth{.50\textwidth}
\begingroup%
  \makeatletter%
  \providecommand\color[2][]{%
    \errmessage{(Inkscape) Color is used for the text in Inkscape, but the package 'color.sty' is not loaded}%
    \renewcommand\color[2][]{}%
  }%
  \providecommand\transparent[1]{%
    \errmessage{(Inkscape) Transparency is used (non-zero) for the text in Inkscape, but the package 'transparent.sty' is not loaded}%
    \renewcommand\transparent[1]{}%
  }%
  \providecommand\rotatebox[2]{#2}%
  \newcommand*\fsize{\dimexpr\f@size pt\relax}%
  \newcommand*\lineheight[1]{\fontsize{\fsize}{#1\fsize}\selectfont}%
  \ifx\svgwidth\undefined%
    \setlength{\unitlength}{127.55905512bp}%
    \ifx\svgscale\undefined%
      \relax%
    \else%
      \setlength{\unitlength}{\unitlength * \real{\svgscale}}%
    \fi%
  \else%
    \setlength{\unitlength}{\svgwidth}%
  \fi%
  \global\let\svgwidth\undefined%
  \global\let\svgscale\undefined%
  \makeatother%
  \begin{picture}(1,0.62222222)%
    \lineheight{1}%
    \setlength\tabcolsep{0pt}%
    \put(0,0){\includegraphics[width=\unitlength,page=1]{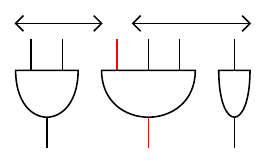}}%
    \put(0.19425114,0.56723111){\color[rgb]{0,0,0}\makebox(0,0)[lt]{\lineheight{1.25}\smash{\begin{tabular}[t]{l}$\kund^l$\end{tabular}}}}%
    \put(0.67954159,0.57208724){\color[rgb]{0,0,0}\makebox(0,0)[lt]{\lineheight{1.25}\smash{\begin{tabular}[t]{l}$\kund^r$\end{tabular}}}}%
    \put(0.43738976,0.56499181){\color[rgb]{0,0,0}\makebox(0,0)[lt]{\lineheight{1.25}\smash{\begin{tabular}[t]{l}$i$\end{tabular}}}}%
  \end{picture}%
\endgroup%

      \caption{The multi-index $(\kund^l|1|\kund^r) = (2,0 |1|2,1)$.}
      \label{fig:bimod_multind}
\end{figure}

We need a little notational preparation for splitting multi-indices. If $\kund$ is a multi-index, and $1\leq i \leq \abs{\kund}$, one can form two multi-indices 
\ea
\kund^l &= (k_1^l, \ldots, k_{a^l}^l, k_{a^l +1}^l ) \\
\kund^r &= (k_0^r,k_1^r, \ldots, k_{a^r}^r )
\ea
as in Figure~\ref{fig:bimod_multind}. The integers $a^l$ and $a^r$ are allowed to be zero, and the entries of $\kund^l$ and $\kund^l$ are $\geq 1$, except $k_{a^l+1}^l$ and $k_0^r$, which are allowed to be zero. The initial multi-index and the splitting positions can be recovered as
\ea
\kund &= (k_1^l, \ldots, k_{a^l+1}^l +1+ k_0^r, \ldots, k_{a^r}^r ),\\
i &= \abs{\kund^l}+1 .
\ea

Write $(\kund^l |1|\kund^r) = (\kund^l_1 |1|\kund^r_1) \sharp (\kund^l_0 |1|\kund^r_0)$ if the associated total multi-indices satisfy $\kund = \kund_1 \sharp \kund_0$, and the splitting positions $i_0, i_1$ are as follows (see Figure~\ref{fig:split_bimod_multind}): we first split $\kund_1$ at position $i_1 = \abs{\kund^l}+1$ to get $\kund_1^l |1|\kund_1^r$, then split $\kund_0$ at position $i_0 = n(\kund_1^l)$ to get $\kund_0^l |1|\kund_0^r$.

We will write $(\kund^l |0|\kund^r)$ if $\kund^l, \kund^r$ are multi-indices with only entries $\geq 1$, and  associate the total multi-index
\e
\kund = (k_1^l, \ldots, k_{a^l}^l , k_1^r, \ldots, k_{a^r}^r ).
\e
We will write $(\kund^l |0|\kund^r) = (\kund^l_1 |0|\kund^r_1) \sharp (\kund^l_0 |0|\kund^r_0)$ if  $\kund^l = \kund^l_1 \sharp \kund^l_0$ and  $\kund^r =\kund^r_1 \sharp \kund^r_0$.

\begin{figure}[!h]
    \centering
    \def\svgwidth{.50\textwidth}
\begingroup%
  \makeatletter%
  \providecommand\color[2][]{%
    \errmessage{(Inkscape) Color is used for the text in Inkscape, but the package 'color.sty' is not loaded}%
    \renewcommand\color[2][]{}%
  }%
  \providecommand\transparent[1]{%
    \errmessage{(Inkscape) Transparency is used (non-zero) for the text in Inkscape, but the package 'transparent.sty' is not loaded}%
    \renewcommand\transparent[1]{}%
  }%
  \providecommand\rotatebox[2]{#2}%
  \newcommand*\fsize{\dimexpr\f@size pt\relax}%
  \newcommand*\lineheight[1]{\fontsize{\fsize}{#1\fsize}\selectfont}%
  \ifx\svgwidth\undefined%
    \setlength{\unitlength}{229.60629921bp}%
    \ifx\svgscale\undefined%
      \relax%
    \else%
      \setlength{\unitlength}{\unitlength * \real{\svgscale}}%
    \fi%
  \else%
    \setlength{\unitlength}{\svgwidth}%
  \fi%
  \global\let\svgwidth\undefined%
  \global\let\svgscale\undefined%
  \makeatother%
  \begin{picture}(1,0.51851852)%
    \lineheight{1}%
    \setlength\tabcolsep{0pt}%
    \put(0,0){\includegraphics[width=\unitlength,page=1]{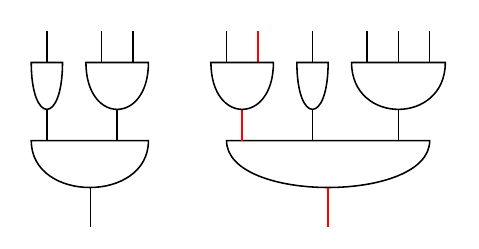}}%
    \put(0.52096208,0.4692016){\color[rgb]{0,0,0}\makebox(0,0)[lt]{\lineheight{1.25}\smash{\begin{tabular}[t]{l}$i_1$\end{tabular}}}}%
    \put(0.51646768,0.24966675){\color[rgb]{0,0,0}\makebox(0,0)[lt]{\lineheight{1.25}\smash{\begin{tabular}[t]{l}$i_0$\end{tabular}}}}%
  \end{picture}%
\endgroup%

      \caption{Splitting a bimodule multi-index.}
      \label{fig:split_bimod_multind}
\end{figure}

Recall that the prefix $u$- stands for up, and reflects the fact that we are splitting the multi-index $\kund$ corresponding to the ascending forest.

\begin{defi}\label{def:kund_bimod} Let  $(A, \left\lbrace {\alpha}^{\kund}_{\lund} \right\rbrace )$ and  $(B,  \left\lbrace {\beta}^{\kund}_{\lund} \right\rbrace )$  be two $f$-bialgebras. An \emph{$(A,B)$  $u$-bimodule (of $f$-bialgebras)} is a chain complex $(M, \partial_M)$, with a collection of operations, with $\epsilon = 0, 1$:
\e
\mu^{\kund^l |\epsilon|\kund^r}_{\lund}\colon (A^b)^{\abs{\kund^l}} \otimes M^{\epsilon b} \otimes (B^b)^{\abs{\kund^r}} \to (A^{\abs{\lund}})^{a^l } \otimes M^{\epsilon \abs{\lund}} \otimes (B^{\abs{\lund}})^{a^r } 
\e
of degree $\deg \mu^{\kund^l |\epsilon|\kund^r}_{\lund} = v(\kund) + v(\lund) -1$, with $\kund$ the total multi-index corresponding to $(\kund^l |\epsilon|\kund^r)$, such that:

\begin{itemize}
\item  they satisfy the family of relations:
\e\label{eq:R_kund_l_eps_r_lund}\tag{$R^{\kund^l |\epsilon|\kund^r}_{\lund}$}
0= \sum_{ } (-1)^{\rho } \cdot \mu^{\kund^l_0 |\epsilon|\kund^r_0}_{\lund_0}\circ \mu^{\kund^l_1 |\epsilon|\kund^r_1}_{\lund_1},
\e
where the sum runs over $(\kund^l |\epsilon|\kund^r,\lund) = (\kund^l_1 |\epsilon|\kund^r_1) \sharp (\kund^l_0 |\epsilon|\kund^r_0)$ and $\lund =\lund_0 \sharp \lund_1$.

\item If $\kund = \One_a$ and $\lund = \One_b$, then $\mu^{\kund^l |\epsilon|\kund^r}_{\lund}$ is the tensor differential  induced by $\alpha^1_1$, $\partial_M =\mu^{0 |1|0}_{1}$ and $\beta^1_1$. We will refer to this condition as $\left(T^{a^l | \epsilon |a^r}_b \right)$.

\item If $\kund^l |\epsilon| \kund^r$ is such that $\kund = \One_a$ and $\lund = (1, \ldots 1, l_j , 1, \ldots, 1)$, with $l_j\geq 2$, then
\e\label{eq:V_al_eps_ar_lund}\tag{$V^{a^l |\epsilon| a^r}_{\lund}$}
{\mu}^{\kund^l |\epsilon| \kund^r}_{\lund} \simeq 
id \otimes{\mu}^{\kund^l |\epsilon| \kund^r}_{l_j} \otimes id,
\e
where $\simeq$ is in the sense of (\ref{eq:iso_split_b}), and $id$ refer to the identities of the appropriate tensor products of $A,M$ and $B$.

\item If  $\lund = \One_b$, then
\e\label{eq:V_kundl_eps_kundr_b}\tag{$V^{\kund^l |\epsilon| \kund^r}_{b}$}
{\mu}^{\kund^l |\epsilon| \kund^r}_{\lund} \simeq 
\begin{cases}
id \otimes {\alpha}^{k_i^l}_{\lund} \otimes id, &\text{ if }\kund^l = (1, \ldots , 1, k_i^l , 1, \ldots , 1);\  \kund^r = \One , \\
id \otimes {\mu}^{k_{a^l +1}^l|1|k_0^r}_{\lund} \otimes id, &\text{ if }(\kund^l |\epsilon| \kund^r) = ( 1, \ldots , 1, k_{a^l +1}^l |1| k_0^r, 1, \ldots , 1), \\
id \otimes {\beta}^{k_i^r}_{\lund} \otimes id, & \text{ if }\kund^l = \One;\  \kund^r = (1, \ldots , 1, k_i^r , 1, \ldots , 1),
\end{cases}
\e
where in each case, $k_i^l$, $ k_{a^l +1}^l + 1+ k_0^r$ and $k_i^r $ respectively are $\geq 2$; $\simeq$ is in the sense of (\ref{eq:iso_split_a}); and $id$ refer to the identities of the appropriate tensor products of $A,M$ and $B$.

\item If $(\kund^l |\epsilon| \kund^r)$ and $\lund$ are such that $c(\kund, \lund) \geq 2$, then ${\mu}^{\kund^l |\epsilon| \kund^r}_{\lund}=0$. We refer to this condition either as $(V^{a^l |\epsilon| a^r}_{\lund})$ or $(V^{\kund^l |\epsilon| \kund^r}_{b})$.

\item When one side is empty,
\e\label{eq:E_kund_l_0_r_lund}\tag{$E^{\kund^l |0|\kund^r}_{\lund}$}
\mu^{\kund^l |0|\kund^r}_{\lund}  = \begin{cases}
\alpha^{\kund^l}_{\lund}\text{ if }\kund^r = (0), \\
\beta^{\kund^r}_{\lund}\text{ if }\kund^l = (0).
\end{cases}
\e

\item (vertical ascending tree deletion) Suppose that $\lund$ is an almost vertical forest, i.e. only with 1 and 2. Assume that $\kund^l |\epsilon|\kund^r$ has a vertical tree at position $i$, depending on what $i$ hits:
\e\label{eq:D_kundl_eps_kundr_i_lund}\tag{$D^{\kund^l |\epsilon|\kund^r,i}_{\lund}$}
\mu^{\kund^l |\epsilon|\kund^r}_{\lund} \simeq \begin{cases}
\mu^{\widehat{\kund}^l |\epsilon|\kund^r}_{\lund} \otimes  \widetilde{\alpha}^{1}_{\lund}\text{ if $i$ hits ${\kund}^l$,}\\
\mu^{\kund^l |\epsilon|\widehat{\kund}^r}_{\lund} \otimes  \widetilde{\beta}^{1}_{\lund}\text{ if $i$ hits ${\kund}^r$,} \\
\mu^{\widehat{\kund}^l |0|\widehat{\kund}^r}_{\lund} \otimes \widetilde{\mu}^{0 |1|0}_{\lund}\text{ if $i$ hits $\epsilon = 1$,}
\end{cases}
\e
where $\widetilde{\alpha}^{1}_{\lund}$ and $\widetilde{\beta}^{1}_{\lund}$ are as in Definition~\ref{def:f_bialg}, and $\widetilde{\mu}^{0 |1|0}_{\lund} = \widetilde{\mu}^{0 |1|0}_{l_1} \otimes \cdots \otimes \widetilde{\mu}^{0 |1|0}_{l_b}$ is defined analogously, with $\widetilde{\mu}^{0 |1|0}_{1}=id_M$ and $\widetilde{\mu}^{0 |1|0}_{2}={\mu}^{0 |1|0}_{2}$. 
In the third case the last (resp. first) entries of $\kund^l$ and $\kund^r$  are equal to 0, and $\widehat{\kund}^l$, $\widehat{\kund}^r$ corresponds to removing them.

\item (vertical descending tree deletion) Assume now that $\kund^l |\epsilon|\kund^r$ is an almost vertical forest, meaning that 
\e
\kund =\begin{cases} (k_1^l, \ldots, k_{a^l+1}^l +1+ k_0^r, \ldots, k_{a^r}^r ) \text{ if }\epsilon = 1,\\ (k_1^l, \ldots, k_{a^l}^l, k_1^r, \ldots, k_{a^r}^r ) \text{ if }\epsilon = 0,
\end{cases}
\e 
is almost vertical. Assume also that $\lund$ has a vertical tree at position $i$, and let $\widehat{\lund}$ with this vertical tree removed. Then:
\e\label{eq:D_kundl_eps_kundr_lund_i}\tag{$D^{\kund^l |\epsilon|\kund^r}_{\lund,i}$}
\mu^{\kund^l |\epsilon|\kund^r}_{\lund} \simeq \mu^{\kund^l |\epsilon|\kund^r}_{\widehat{\lund}}  \otimes  \mu^{\kund^l |\epsilon|\kund^r}_{1}.
\e

\end{itemize}
\end{defi}

Dually, one can define ``$d$-bimodules'' as a family of operations $\mu^{\kund}_{\lund^l |\nu|\lund^r}$ satisfying analogous relations. More generally, one can make sense of ``quadrimodules'':
\begin{defi}\label{def:quadrimod}
Let:
\begin{itemize}
\item four $f$-bialgebras $A,B,C,D$,
\item two $u$-bimodules: an $(A,B)$-bimodule $M$ and a $(C,D)$-bimodule $N$,
\item two $d$-bimodules: an $(A,C)$-bimodule $P$ and a $(B,D)$-bimodule $Q$,
\end{itemize}  
We write all these data in a box
\e
\square := \begin{pmatrix}
A & M & B \\ 
P & . & Q \\ 
C & N & D
\end{pmatrix} ,
\e
and say that $R$ is a \emph{$\square$-quadrimodule} if it is a chain complex with a family of operations, with $\epsilon, \nu = 0,1$: 
\ea
\mu^{\kund^l |\epsilon|\kund^r}_{\lund^l |\nu|\lund^r} &\colon \left( A^{b^l} {P}^{\nu} C^{b^r} \right)^{\abs{\kund^l}} \otimes \left( M^{b^l} R^{\nu} N^{b^r} \right)^{\epsilon} \otimes \left(B^{b^l} {Q}^{\nu} D^{b^r} \right)^{\abs{\kund^r}} \\
 &\to \left( A^{\abs{\lund^l}} {P}^{\nu} C^{\abs{\lund^r}} \right)^{a^l} \otimes\left(  M^{\abs{\lund^l}} R^{\nu} N^{\abs{\lund^r}} \right)^{\epsilon} \otimes \left(  B^{\abs{\lund^l}} {Q}^{\nu} D^{\abs{\lund^r}}\right)^{a^r} \nonumber
\ea
Such that:
\begin{itemize}
\item When ${\kund^l |\epsilon|\kund^r} = {\kund^l |0|0}$ (resp. ${\kund^l |\epsilon|\kund^r} = { 0|0|\kund^r}$), $\mu^{\kund^l |\epsilon|\kund^r}_{\lund^l |\nu|\lund^r}$ equals to the corresponding structure operation of $P$ (resp. $Q$),
\item When ${\lund^l |\nu|\lund^r} = {\lund^l |0|0}$ (resp. ${\lund^l |\nu|\lund^r} = { 0|0|\lund^r}$), $\mu^{\kund^l |\epsilon|\kund^r}_{\lund^l |\nu|\lund^r}$ equals to the corresponding structure operation of $M$ (resp. $N$),
\item they satisfy  coherence and simplification relations
\e\label{eq:quadrimod}
 (T^{a^l |\epsilon|a^r}_{b^l |\nu|b^r}),\  (R^{\kund^l |\epsilon|\kund^r}_{\lund^l |\nu|\lund^r}),\ (V^{\kund^l |\epsilon|\kund^r}_{b^l|\nu|b^r}),\ (V^{a^l |\epsilon|a^r}_{\lund^l |\nu|\lund^r}),\ (D^{\kund^l |\epsilon|\kund^r,i}_{\lund^l |\nu|\lund^r}),\ (D^{\kund^l |\epsilon|\kund^r}_{\lund^l |\nu|\lund^r,i})
\e
analogous to the ones for bimodules.
\end{itemize}
\end{defi}

\subsection{Morphisms of bimodules of $f$-bialgebras}
\label{ssec:mph_bimod_Ainf_bialg}
The following is the notion of morphisms appearing in Conjecture~\ref{conj:Morse}:

\begin{defi}\label{def:mph_bimod} Let:

\noindent\begin{minipage}{.2\textwidth}\begin{tikzcd}
   A \ar{r}{f}& C \\
   M \ar{r}{\Phi} \ar[,loop ,out=123,in=57,distance=2.5em]{}{} \ar[,loop ,out=-123,in=-57,distance=2.5em]{}{}& N \ar[,loop ,out=123,in=57,distance=2.5em]{}{} \ar[,loop ,out=-123,in=-57,distance=2.5em]{}{}\\
   B \ar{r}{g}& D
\end{tikzcd}
\end{minipage}
\begin{minipage}{.7\textwidth}
\begin{itemize}
\item  $A,B,C,D$ be four $f$-bialgebras,
\item $(M, \mu)$ be an $(A,B)$-bimodule, 
\item $(N, \nu)$ be a $(C,D)$-bimodule, 
\end{itemize}
\end{minipage}

A \emph{morphism of bimodules} $(f,\Phi,g) \colon M\to N$  consists in 
two morphisms of $f$-bialgebras, $f\colon A \to C$, $g\colon B \to D$, with a family of operations
\e
\Phi^{\kund^l |\epsilon|\kund^r}_{\lund}\colon (A^b)^{\abs{\kund^l}} \otimes (M^b)^\epsilon \otimes (B^b)^{\abs{\kund^r}} \to (C^{\abs{\lund}})^{a^l } \otimes (N^{\abs{\lund}})^\epsilon \otimes (D^{\abs{\lund}})^{a^r } 
\e
Satisfying:
\e\tag{$ M^{\kund^l |\epsilon|\kund^r}_{\lund}$ }
0 = \sum_{}{ (-1)^{\rho_0} \cdot \Phi^{\kund^l_0 |\epsilon|\kund^r_0}_{\lund_0}  \circ  \mu^{\kund^l_1 |\epsilon|\kund^r_1}_{\lund_1}   + (-1)^{\rho_1} \cdot  \nu^{\kund^l_0 |\epsilon|\kund^r_0}_{\lund_0}\circ \Phi^{\kund^l_1 |\epsilon|\kund^r_1}_{\lund_1}   },
\e
where the sum runs over $(\kund^l |\epsilon|\kund^r) = (\kund^l_1 |\epsilon|\kund^r_1) \sharp (\kund^l_0 |\epsilon|\kund^r_0)$ and $ \lund = \lund^0 \sharp \lund^1$. And also:
\begin{itemize}
\item If $\lund$ is vertical,
\ea\tag{$ W^{\kund^l |0|\kund^r}_{b}$ }
\Phi^{\kund^l |0|\kund^r}_{\lund} &\simeq f^{k^l_1}_{\lund}   \otimes \cdots \otimes  f^{k^l_{a^l }}_{\lund} \otimes  g^{k^r_1}_{\lund}  \otimes \cdots \otimes  g^{k^r_{a^r}}_{\lund}, \\
\tag{$ W^{\kund^l |1|\kund^r}_{b}$ }
\Phi^{\kund^l |1|\kund^r}_{\lund} &\simeq f^{k^l_1}_{\lund}   \otimes \cdots \otimes  f^{k^l_{a^l }}_{\lund} \otimes \Phi^{k^l_{a^l +1} |1|k^r_0}_{\lund}\otimes g^{k^r_1}_{\lund}  \otimes \cdots \otimes  g^{k^r_{a^r}}_{\lund} .
\ea

\item If $\kund^l |\epsilon|\kund^r$ is vertical, 
\e\tag{$ W^{a^l |\epsilon|a^r}_{\lund}$ }
\Phi^{\kund^l |\epsilon|\kund^r}_{\lund} \simeq \Phi^{\kund^l |\epsilon|\kund^r}_{l_1}  \otimes \cdots \otimes \Phi^{\kund^l |\epsilon|\kund^r}_{l_b} .
\e

\end{itemize}

\end{defi}

\subsection{Bimodule categories and functors}
\label{ssec:bimod_categ_fun}

Here we define the categorical notions appearing in Conjecture~\ref{conj:Floer}. It accounts for how an $f$-bialgebra (say the Morse complex $CM(G)$) acts on an \Ainf -(co)category (the $G$-invariant part of the Fukaya category $\Fuk(M)$). Loosely speaking, we are adding objects to a $u$-bimodule. The action is at the level of morphisms only. The more general notion, where the action is also on objects, will be introduced in Section~\ref{ssec:lund_categ_bimod_fun}.

First, some notational preliminaries. We will need to decorate multi-indices by objects. Recall from Section~\ref{ssec:geom_realiz} that to a biforest $(U,D)$ in $K^{\kund}_{\lund}$ we associate a hybrid graph 
\e
H(U,D) = \abs{U}^{\mathrm{Morse}} \times_{\rr}  \langle D \rangle \cup \abs{U}^{\mathrm{Floer}}\times_{\rr} \langle\!\langle D \rangle\!\rangle.
\e 
To form moduli spaces from these hybrid graphs, we need to prescribe some Lagrangian boundary conditions on the Riemann surface part $\abs{U}^{\mathrm{Floer}}\times_{\rr} \langle\!\langle D \rangle\!\rangle$. This is what our decorated multi-indices are doing. 

\begin{defi}\label{def:deco_multind} Let $\mathrm{Ob}$ be a set. We say $\Lund$ is a \emph{multi-index decorated by} $\mathrm{Ob}$ if it is  a multi-index $\lund = (l_1, \ldots, l_b)$, together with a family of $\abs{\lund} + n(\lund)$ elements of $\mathrm{Ob}$:
\e
L_{1,0}, \ldots ,L_{1,l_1} , L_{2, 0}, \ldots ,L_{2,l_2} , L_{3,0}, \ldots ,L_{b,l_b}.
\e
These correspond to the boundary components of the surface in Figure~\ref{fig:dec_lund}. 
\end{defi}

\begin{figure}[!h]
    \centering
    \def\svgwidth{.70\textwidth}
\begingroup%
  \makeatletter%
  \providecommand\color[2][]{%
    \errmessage{(Inkscape) Color is used for the text in Inkscape, but the package 'color.sty' is not loaded}%
    \renewcommand\color[2][]{}%
  }%
  \providecommand\transparent[1]{%
    \errmessage{(Inkscape) Transparency is used (non-zero) for the text in Inkscape, but the package 'transparent.sty' is not loaded}%
    \renewcommand\transparent[1]{}%
  }%
  \providecommand\rotatebox[2]{#2}%
  \newcommand*\fsize{\dimexpr\f@size pt\relax}%
  \newcommand*\lineheight[1]{\fontsize{\fsize}{#1\fsize}\selectfont}%
  \ifx\svgwidth\undefined%
    \setlength{\unitlength}{368.50393701bp}%
    \ifx\svgscale\undefined%
      \relax%
    \else%
      \setlength{\unitlength}{\unitlength * \real{\svgscale}}%
    \fi%
  \else%
    \setlength{\unitlength}{\svgwidth}%
  \fi%
  \global\let\svgwidth\undefined%
  \global\let\svgscale\undefined%
  \makeatother%
  \begin{picture}(1,0.53846154)%
    \lineheight{1}%
    \setlength\tabcolsep{0pt}%
    \put(0,0){\includegraphics[width=\unitlength,page=1]{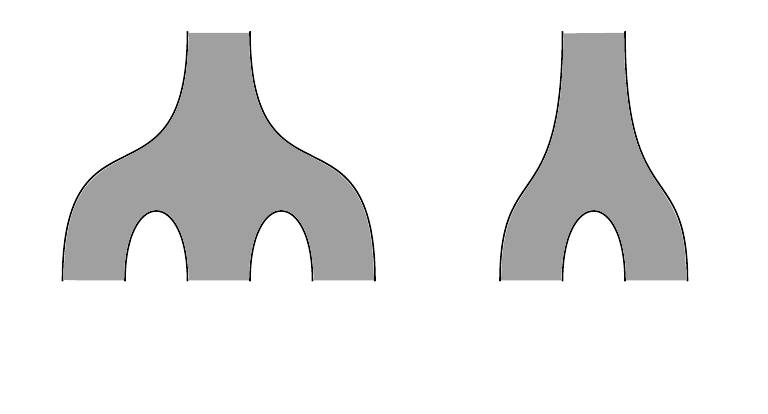}}%
    \put(0.01965682,0.13415693){\color[rgb]{0,0,0}\makebox(0,0)[lt]{\lineheight{1.25}\smash{\begin{tabular}[t]{l}$L_{1,0}$\end{tabular}}}}%
    \put(0.16488011,0.13446548){\color[rgb]{0,0,0}\makebox(0,0)[lt]{\lineheight{1.25}\smash{\begin{tabular}[t]{l}$L_{1,1}$\end{tabular}}}}%
    \put(0.32800456,0.1354807){\color[rgb]{0,0,0}\makebox(0,0)[lt]{\lineheight{1.25}\smash{\begin{tabular}[t]{l}$L_{1,2}$\end{tabular}}}}%
    \put(0.48530263,0.13271003){\color[rgb]{0,0,0}\makebox(0,0)[lt]{\lineheight{1.25}\smash{\begin{tabular}[t]{l}$L_{1,3}$\end{tabular}}}}%
    \put(0.61602678,0.13422649){\color[rgb]{0,0,0}\makebox(0,0)[lt]{\lineheight{1.25}\smash{\begin{tabular}[t]{l}$L_{2,0}$\end{tabular}}}}%
    \put(0.73124188,0.1326283){\color[rgb]{0,0,0}\makebox(0,0)[lt]{\lineheight{1.25}\smash{\begin{tabular}[t]{l}$L_{2,1}$\end{tabular}}}}%
    \put(0.89624618,0.13367875){\color[rgb]{0,0,0}\makebox(0,0)[lt]{\lineheight{1.25}\smash{\begin{tabular}[t]{l}$L_{2,2}$\end{tabular}}}}%
  \end{picture}%
\endgroup%

      \caption{A decoration of $\lund = (3,2)$.}
      \label{fig:dec_lund}
\end{figure}

If $\lund = \lund^0 \sharp \lund^1$, a decoration $\Lund$ of $\lund$ induces decorations $\Lund^0, \Lund^1$ of $\lund^0 \sharp \lund^1$. In this case we will write $\Lund = \Lund^0 \sharp \Lund^1$. This is obvious from Figure~\ref{fig:break_dec}, although writing these explicitly requires a few notational preliminaries, detailed below.

\begin{figure}[!h]
    \centering
    \def\svgwidth{.50\textwidth}
\begingroup%
  \makeatletter%
  \providecommand\color[2][]{%
    \errmessage{(Inkscape) Color is used for the text in Inkscape, but the package 'color.sty' is not loaded}%
    \renewcommand\color[2][]{}%
  }%
  \providecommand\transparent[1]{%
    \errmessage{(Inkscape) Transparency is used (non-zero) for the text in Inkscape, but the package 'transparent.sty' is not loaded}%
    \renewcommand\transparent[1]{}%
  }%
  \providecommand\rotatebox[2]{#2}%
  \newcommand*\fsize{\dimexpr\f@size pt\relax}%
  \newcommand*\lineheight[1]{\fontsize{\fsize}{#1\fsize}\selectfont}%
  \ifx\svgwidth\undefined%
    \setlength{\unitlength}{368.50393701bp}%
    \ifx\svgscale\undefined%
      \relax%
    \else%
      \setlength{\unitlength}{\unitlength * \real{\svgscale}}%
    \fi%
  \else%
    \setlength{\unitlength}{\svgwidth}%
  \fi%
  \global\let\svgwidth\undefined%
  \global\let\svgscale\undefined%
  \makeatother%
  \begin{picture}(1,0.84615385)%
    \lineheight{1}%
    \setlength\tabcolsep{0pt}%
    \put(0,0){\includegraphics[width=\unitlength,page=1]{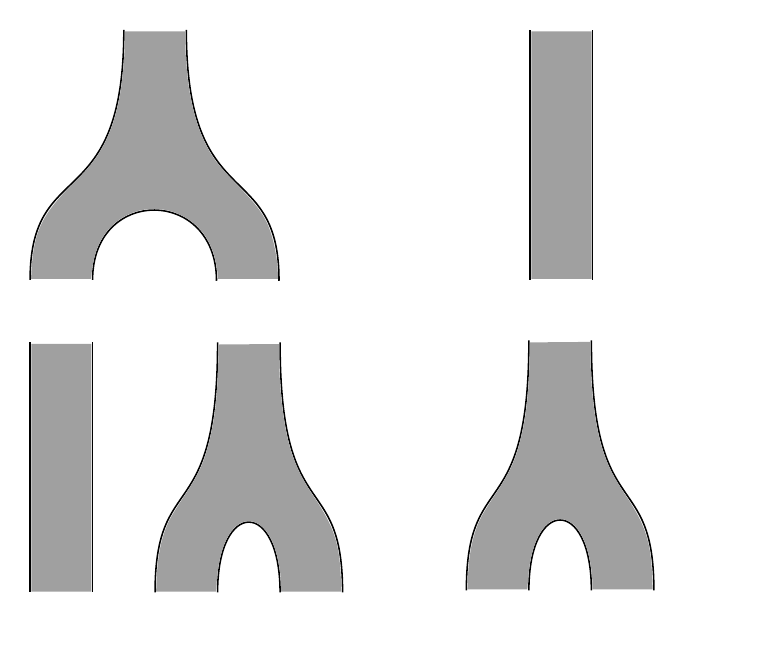}}%
    \put(0.89551282,0.62227564){\color[rgb]{0,0,0}\makebox(0,0)[lt]{\lineheight{1.25}\smash{\begin{tabular}[t]{l}$\Lund^1$\end{tabular}}}}%
    \put(0.89551282,0.2152243){\color[rgb]{0,0,0}\makebox(0,0)[lt]{\lineheight{1.25}\smash{\begin{tabular}[t]{l}$\Lund^0$\end{tabular}}}}%
  \end{picture}%
\endgroup%

      \caption{Splitting a decoration $\Lund$.}
      \label{fig:break_dec}
\end{figure}

For a multi-index $\lund = (l_1, \ldots , l_b)$, and a pair of integers $(j,j')$ such that $1\leq j \leq b$ and  $0\leq j' \leq l_j$, recall that we introduced
\e
s_{\lund}(j,j') = l_1 + \cdots + l_{j-1} + j'.
\e
Notice that here we allow $j'=0$, in which case (if $j\geq 2$) one has $s_{\lund}(j,0) = s_{\lund}(j-1,l_{j-1}) $. 
With these notations, let us define $\Lund^0$ by setting $L^0_{j,j'} = L_{h,h'}$, where:
\begin{itemize}
\item $h$ is such that $j = s_{\lund}(h,\tilde{h})$, for some $\tilde{h}$,
\item let $\lund^{0;h}$ correspond to the subforest of $\lund^0$ growing on top of the $h$-th tree of $\lund^1$, namely:
\e
\lund^{0;h} = (l^0_{l_1 + \cdots + l_{h-1} + 1}, l^0_{l_1 + \cdots + l_{h-1} + 2}, \ldots , l^0_{l_1 + \cdots + l_{h}} ).
\e
\item let $j^h = j- (l_1 + \cdots + l_{h-1})$, so that the $j$-th root of $\lund^0$ corresponds to the $j^h$-th root of $\lund^{0;h}$.
\item let finally $h' = s_{\lund^{0;h}} (j^h, j')$.
\end{itemize}

Finally, $\Lund^1$ is more easily defined by:
\e
L^1_{j,j'} = L_{j, l^0_1 + \cdots + l^0_{j'}}.
\e

We are now ready to define ``$u$-bimodule $d$-categories'', relevant to Conjecture~\ref{conj:Floer}. Here, the prefix $d$- indicates that objects decorate the multi-index $\lund$ of the descending forests.

\begin{defi}Let $A,B$ be two $f$-bialgebras. An $(A,B)$ $u$-bimodule $d$-category $\Mcal$ consists in
\begin{itemize}
\item A set of objects $\mathrm{Ob}(\Mcal)$,
\item For two objects $L_0, L_1$, a chain complex $\Mcal (L_0, L_1)$
\item A family of operations satisfying a family of relations, that we describe below. 
\end{itemize}
Let us first introduce, for $\Lund$ a decoration of $\lund$ by $\mathrm{Ob}(\Mcal)$:
\ea
\Mcal(\Lund)^{in}  &:= \bigotimes_{j=1}^b \Mcal(L_{j,0}, L_{j,l_j}) ,  \\
\Mcal(\Lund)^{out} &:= \bigotimes_{j=1}^b \bigotimes_{j'=1}^{l_j}  \Mcal(L_{j,j'-1}, L_{j,j'}).
\ea
The operations  $\mu^{\kund^l |\epsilon|\kund^r}_{\Lund}$, with $\epsilon=0$ or $1$, are of the form:
\ea
\mu^{\kund^l |\epsilon|\kund^r}_{\Lund} &\colon (A^b)^{\abs{\kund^l}} \otimes \left( \Mcal(\Lund)^{in} \right)^{\epsilon} \otimes (B^b)^{\abs{\kund^r}}\\
&\to   (A^{\abs{\lund}})^{a^l} \otimes \left( \Mcal(\Lund)^{out} \right)^{\epsilon} \otimes (B^{\abs{\lund}})^{a^r} \nonumber
\ea

\begin{figure}[!h]
    \centering
    \def\svgwidth{.50\textwidth}
    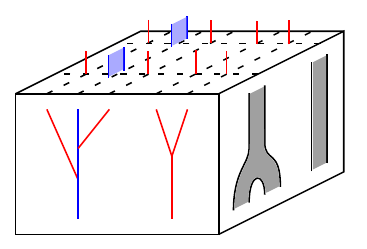
      \caption{The map $\mu^{\kund^l |1|\kund^r}_{\Lund}$ as a rectangular box.}
      \label{fig:rectangular_box_cat}
\end{figure}
For $\epsilon = 1$, one can think of $\mu^{\kund^l |1|\kund^r}_{\Lund}$ as a rectangular box as in Figure~\ref{fig:rectangular_box_cat}. 
For $\epsilon=0$ they only depend on $\Lund$ through $\lund$, so we also denote them $\mu^{\kund^l |0|\kund^r}_{\lund}$. We may also do so when $\epsilon=1$, although this is an abuse of notation (as for \Ainf -categories).
These operations are required to satisfy the obvious analogs 
\[
(R^{\kund^l |\epsilon|\kund^r}_{\Lund}),\ (T^{a^l |\epsilon|a^r}_{\Lund}),\ (V^{a^l |\epsilon|a^r}_{\Lund}),\ (V^{\kund^l |\epsilon|\kund^r}_{b}), \ (E^{\kund^l |0|\kund^r}_{\Lund}),\  (D^{\kund^l |\epsilon|\kund^r,i}_{\Lund}),\ (D^{\kund^l |\epsilon|\kund^r}_{\Lund,i})
\] 
of equations  (\ref{eq:R_kund_l_eps_r_lund}), $\left(T^{a^l | \epsilon |a^r}_b \right)$, ( \ref{eq:V_al_eps_ar_lund}), (\ref{eq:V_kundl_eps_kundr_b}),  (\ref{eq:E_kund_l_0_r_lund}),(\ref{eq:D_kundl_eps_kundr_i_lund}), (\ref{eq:D_kundl_eps_kundr_lund_i}),
%
%
i.e. replacing multi-indices $\lund, \lund^0, \lund^1$ by decorated multi-indices $\Lund, \Lund^0, \Lund^1$.
\end{defi}

Such $u$-bimodule $d$-categories come with a notion of functors, based on morphisms of $u$-bimodules:
\begin{defi}
Let

\noindent\begin{minipage}{.2\textwidth}\begin{tikzcd}
   A \ar{r}{f}& C \\
   \Mcal \ar{r}{\Phi} \ar[,loop ,out=123,in=57,distance=2.5em]{}{} \ar[,loop ,out=-123,in=-57,distance=2.5em]{}{}& \Ncal \ar[,loop ,out=123,in=57,distance=2.5em]{}{} \ar[,loop ,out=-123,in=-57,distance=2.5em]{}{}\\
   B \ar{r}{g}& D
\end{tikzcd}
\end{minipage}
\begin{minipage}{.7\textwidth}
\begin{itemize}
\item  $A,B,C,D$ be four $f$-bialgebras,
\item $f\colon A\to C$ and $g \colon B\to D$ be two morphisms of $f$-bialgebras,
\item $\Mcal$ be  an $(A,B)$  $u$-bimodule $d$-category.
\item  $\Ncal$ be  a $(C,D)$ $u$-bimodule $d$-category.
\end{itemize}
\end{minipage}

Then a \emph{functor} $\Phi\colon \Mcal \to \Ncal$ over $(f,g)$ consists in a map on objects $\Phi_{\mathrm{ob}}\colon \mathrm{Ob} (\Mcal) \to \mathrm{Ob}(\Ncal)$, and a collections of maps:
\ea
\Phi^{\kund^l |\epsilon|\kund^r}_{\Lund} &\colon (A^b)^{\abs{\kund^l}} \otimes \left( \Mcal(\Lund)^{in} \right)^{\epsilon} \otimes (B^b)^{\abs{\kund^r}}\\
&\to   (C^{\abs{\lund}})^{a^l} \otimes \left( \Ncal(\Phi_{\mathrm{ob}}(\Lund))^{out} \right)^{\epsilon} \otimes (D^{\abs{\lund}})^{a^r} \nonumber
\ea
satisfying the obvious analogs of the (M) and (W) 
relations of Definition~\ref{def:mph_bimod}.
\end{defi}

\subsection{$d$-categories, bimodules and functors}
\label{ssec:lund_categ_bimod_fun}
We now turn to the more general setting of Conjecture~\ref{conj:wrapped}, where the $f$-bialgebras become categories. This involves more general decorations. Here, the relevant domain is the foam $F(U, D) = \abs{U}\times_{\rr} \langle\!\langle D \rangle\!\rangle$, and we want to decorate its ``external'' boundary components, i.e. those of $\abs{U}\times_{\rr} \partial \langle\!\langle D \rangle\!\rangle$. The other components $\partial\abs{U}\times_{\rr}  \langle\!\langle D \rangle\!\rangle$ are ``seams'', and will be decorated with the Weinstein correspondences $\Lambda_G(M) \subset T^* G \times M^- \times M$ (resp. $\Lambda_G(T^* G) \subset T^* G \times T^* G^- \times T^* G$) that encodes the Hamiltonian action and moment map (resp. the group structure). 

To ensure ``continuity of the moduli spaces'' at graph transitions, the rule is that at a vertex $v$ of $U$, the incoming decorations determine the outgoing one. If the boundaries corresponding to incoming edges are decorated by $L\subset T^*G$ and $L' \subset M$, the outgoing boundary is decorated by the composition of correspondences $\Lambda_G(M) \circ (L\times L') \subset M$.

\begin{defi}\label{def:dec_kund_lund} Let $\mathrm{Ob}$ be a monoid. A \emph{decoration of} $(\kund,\lund)$ \emph{by} $\mathrm{Ob}$ is a collection of elements $\Lbb = \lbrace L^{i,i'}_{j,j'}\rbrace$ in $\mathrm{Ob}$, with:
\ea
1\leq i \leq a &,\ 1\leq i' \leq k_i ,\label{eq:i_i'} \\
1\leq j \leq b &,\ 0\leq j' \leq l_j .
\ea
Notice that $i'$ starts at 1, whereas $j'$ starts at 0. Equivalently, it is a collection $\Lund^{i,i'}$ of decorations of $\lund$, in the sense of Definition~\ref{def:deco_multind}.
\end{defi}

We will associate to such $\Lbb$ another collection $\Lbb^{\times} = \lbrace L^{i,\times}_{j,j'}\rbrace$ given by multiplying leaves of the ascending trees of $\kund$ together, i.e.
\e
L^{i,\times}_{j,j'} = L^{i,1}_{j,j'} \times \cdots \times L^{i,k_i}_{j,j'}.
\e

If $\kund = \kund^1 \sharp \kund^0$ and $\lund = \lund^0 \sharp \lund^1$, which we denote $(\kund, \lund) = (\kund^1,\lund^1) \sharp (\kund^0,\lund^0)$ a decoration $\Lbb$ of $(\kund, \lund)$ induces decorations $\Lbb^1, \Lbb^0$ of $(\kund^1,\lund^1)$, $(\kund^0,\lund^0)$, as we explain below. We will write $\Lbb = \Lbb^1 \sharp \Lbb^0$. Notice first that pairs $(i,i')$ satisfying (\ref{eq:i_i'}) are in bijection with $\lbrace 1, \ldots , \abs{\kund} \rbrace$ through the map $s_{\kund}$. Since $\abs{\kund}  =\abs{\kund^1} $, by applying $(s_{\kund^1})^{-1}$ one gets a pair $(\tilde{i},\tilde{i}')$ satisfying (\ref{eq:i_i'}) for $\kund^1$. Therefore, the decoration $\Lbb$ can be thought of as a decoration $\widetilde{\Lbb}$ of $(\kund^1,\lund)$. Applying the decoration constructions of Section~\ref{ssec:bimod_categ_fun} to each $\Lund^{\tilde{i},\tilde{i}'}$ individually then yields the decoration $\Lbb^1$ of $(\kund^1,\lund^1)$. 

Consider now $\widetilde{\Lbb}^{\times}$. Since $n(\kund^1) = \abs{\kund^0}$, it can be seen as a decoration of $(\kund^0,\lund)$. Applying again the decoration constructions  to  $\Lund^{\tilde{i},\times}$ then yields the decoration $\Lbb^0$ of $(\kund^0,\lund^0)$. 
Define then the (uncurved version of the) structure that $\Wcal(T^*G)$ should have, according to Conjecture~\ref{conj:wrapped}.

\begin{defi}\label{def:lund_cat} An \emph{($f$-bialgebra) $d$-category} $\Acal$ consists in:
\begin{itemize}
\item A monoid of objects $\mathrm{Ob}(\Acal)$, 
\item For each pair of objects $L,L' \in \mathrm{Ob}(\Acal)$, a chain complex $\Acal (L,L')$,
\item A collection of operations satisfying a collection of relations, detailed below.
\end{itemize}

Let $\Lbb$ be a decoration of $(\kund, \lund)$, define 
\ea
\mathrm{In}(\Lbb) &= \bigotimes_{i=1}^{a} \bigotimes_{i'=1}^{k_i} \bigotimes_{j=1}^{b} \Acal (L^{i,i'}_{j,0}, L^{i,i'}_{j,l_j}) \\
 &= \bigotimes_{i=1}^{a} \bigotimes_{i'=1}^{k_i}  \Acal (\Lund^{i,i'})^{in}, \nonumber \\
\mathrm{Out}(\Lbb) &= \bigotimes_{i=1}^{a}  \bigotimes_{j=1}^{b} \bigotimes_{j'=1}^{l_j}\Acal (L^{i,\times}_{j,j'-1}, L^{i,\times}_{j,j'}) \\
 &= \bigotimes_{i=1}^{a} \Acal (\Lund^{i,\times})^{out}. \nonumber
\ea
The operations are then denoted 
\e
\alpha (\Lbb)^{\kund}_{\lund}\colon \mathrm{In}(\Lbb) \to \mathrm{Out}(\Lbb),
\e
or just $\alpha (\Lbb)$  or $\alpha^{\kund}_{\lund}$, depending on what is convenient or implicit.

Assume $\Lbb = \Lbb^1 \sharp \Lbb^0$, and notice that 
\ea
\mathrm{In}(\Lbb) &= \mathrm{In}(\Lbb^1) ,\\
\mathrm{Out}(\Lbb^1) &= \mathrm{In}(\Lbb^0) ,\\
\mathrm{Out}(\Lbb) &= \mathrm{Out}(\Lbb^0) .
\ea
Therefore  $\alpha (\Lbb^0)$ and  $\alpha (\Lbb^1)$ can be composed, and $\alpha (\Lbb^0)\circ \alpha (\Lbb^1)$ goes from $\mathrm{In}(\Lbb)$ to $\mathrm{Out}(\Lbb)$ just as $\alpha (\Lbb)$.

These operations are required to satisfy the obvious analogs of the $f$-bialgebra relations in Definition~\ref{def:f_bialg}, that we will denote
\e\nonumber
(R(\Lbb)^{\kund}_{\lund}),\ (T(\Lbb)^{a}_{b}),\ (V(\Lbb)^{\kund}_{b}),\ (V(\Lbb)^{a}_{\lund}),\ (D(\Lbb)^{\kund;i}_{\lund}),\ (D(\Lbb)^{\kund}_{\lund;i}).
\e
\end{defi}

\begin{defi}\label{def:kund_bimod_over_lund_cats} 
Let $\Acal$, $\Bcal$ be two $d$-categories. An \emph{$(\Acal,\Bcal)$ $u$-bimodule $d$-category} $\Mcal$ consists in
\begin{itemize}
\item A set of objects $\mathrm{Ob}(\Mcal)$, endowed with two commuting actions 
\ea
\mathrm{Ob}(\Acal) \times \mathrm{Ob}(\Mcal) \to \mathrm{Ob}(\Mcal), \\
\mathrm{Ob}(\Mcal) \times \mathrm{Ob}(\Bcal) \to \mathrm{Ob}(\Mcal).
\ea
\item For each pair of objects $L,L' \in \mathrm{Ob}(\Mcal)$, a chain complex $\Mcal (L,L')$,
\item A collection of operations satisfying a collection of relations, detailed below.
\end{itemize}
We first need to decorate $u$-bimodule indices. A \emph{decoration of} $(\kund, \lund) = (\kund^l |\epsilon|\kund^r, \lund)$ is a collection $\Lbb = (\Lbb^l | \Lund^\epsilon | \Lbb^r) =  \lbrace L^{i,i'}_{j,j'}\rbrace$, where:
\begin{itemize}
\item $\Lbb^l$ is a decoration of  $(\kund^l, \lund)$ by $\mathrm{Ob}(\Acal)$,
\item $\Lund^\epsilon$ is either empty if $\epsilon = 0$, or a decoration of $\lund$ by $\mathrm{Ob}(\Mcal)$ if $\epsilon=1$,
\item $\Lbb^r$ is a decoration of  $(\kund^r, \lund)$ by $\mathrm{Ob}(\Bcal)$.
\end{itemize}
As before, $\Lbb$ induces sub-decorations $\Lbb^0, \Lbb^1$ when splitting multi-indices, and the input and output spaces are 
\ea
\mathrm{In}(\Lbb) &= \bigotimes_{i=1}^{a} \bigotimes_{i'=1}^{k_i} \bigotimes_{j=1}^{b} \Xcal (L^{i,i'}_{j,0}, L^{i,i'}_{j,l_j}), \\
\mathrm{Out}(\Lbb) &= \bigotimes_{i=1}^{a}  \bigotimes_{j=1}^{b} \bigotimes_{j'=1}^{l_j}\Xcal (L^{i,\times}_{j,j'-1}, L^{i,\times}_{j,j'}), 
\ea
where $\Xcal$ stands either for $\Acal$, $\Mcal$ or $\Bcal$.

The operations are then denoted 
\e
\mu (\Lbb)^{\kund^l |\epsilon|\kund^r}_{\lund}\colon \mathrm{In}(\Lbb) \to \mathrm{Out}(\Lbb),
\e
or just $\alpha (\Lbb)$,  or $\alpha^{\kund}_{\lund}$, etc...

These operations are required to satisfy the obvious analogs of the $u$-bimodules relations in Definition~\ref{def:kund_bimod}.
\end{defi}

Finally, $d$-categories and their $u$-bimodules come with their functors:
\begin{defi}\label{def:fun_lund_cat} A \emph{functor between $d$-categories $\Fcal\colon \Acal \to \Bcal$} consists in
\begin{itemize}
\item a morphism of monoids $\Fcal_{\mathrm{ob}}\colon \mathrm{Ob}(\Acal) \to \mathrm{Ob}(\Bcal)$,
\item for decorations $\Lbb$ of $(\kund,\lund)$, maps $\Fcal(\Lbb)^{\kund}_{\lund} \colon \mathrm{In}(\Lbb) \to \mathrm{Out}(\Fcal_{\mathrm{ob}}(\Lbb))$ satisfying analogous relations of $f$-bialgebra morphisms, as in Definition~\ref{def:mph_f_bialg}.
\end{itemize}

\end{defi}

\begin{defi}\label{def:fun_kund_bimod_lund_cat} 
Let: 

\noindent\begin{minipage}{.2\textwidth}\begin{tikzcd}
   \Acal \ar{r}{\Fcal}& \Ccal \\
   \Mcal \ar{r}{\Phi} \ar[,loop ,out=123,in=57,distance=2.5em]{}{} \ar[,loop ,out=-123,in=-57,distance=2.5em]{}{}& \Ncal \ar[,loop ,out=123,in=57,distance=2.5em]{}{} \ar[,loop ,out=-123,in=-57,distance=2.5em]{}{}\\
   \Bcal \ar{r}{\Gcal}& \Dcal
\end{tikzcd}
\end{minipage}
\begin{minipage}{.7\textwidth}
\begin{itemize}
\item  $\Acal$, $\Bcal$, $\Ccal$, $\Dcal$ be four $d$-categories,
\item $\Fcal\colon \Acal \to \Ccal$ and $\Gcal\colon \Bcal \to \Dcal$ be functors of $d$-categories,
\item $\Mcal$, be  an $(\Acal, \Bcal)$  $u$-bimodule $d$-category.
\item  $\Ncal$ be  a $(\Ccal, \Dcal)$ $u$-bimodule $d$-category.
\end{itemize}
\end{minipage}

An \emph{$(\Fcal, \Gcal)$-equivariant functor between $u$-bimodules $d$-categories $\Phi\colon \Mcal \to \Ncal$} consists in
\begin{itemize}
\item a $(\Fcal_{\mathrm{ob}}, \Gcal_{\mathrm{ob}})$-equivariant map
$\Phi_{\mathrm{ob}}\colon \mathrm{Ob}(\Mcal) \to \mathrm{Ob}(\Ncal)$,

\item for decorations $\Lbb$ of $(\kund^l |\epsilon|\kund^r,\lund)$, maps $\Phi(\Lbb)^{\kund^l |\epsilon|\kund^r}_{\lund} \colon \mathrm{In}(\Lbb) \to \mathrm{Out}(\Phi_{\mathrm{ob}}(\Lbb))$ satisfying analogous relations of $u$-bimodules morphisms, as in Definition~\ref{def:mph_bimod}. Here, if $\Lbb = (\Lbb^l | \Lund^\epsilon | \Lbb^r)$, we denote $\Phi_{\mathrm{ob}}(\Lbb) = (\Fcal_{\mathrm{ob}}(\Lbb^l) | \Phi_{\mathrm{ob}}(\Lund^\epsilon)| \Gcal_{\mathrm{ob}}(\Lbb^r) )$.
\end{itemize}

\end{defi}

\bibliographystyle{alpha}
\bibliography{biblio}

\end{document}